\documentclass[final,leqno,showlabe]{siamltex}
\usepackage{amsmath}
\usepackage{stmaryrd}
\usepackage{booktabs}
\usepackage{cite}
\usepackage{amssymb}
\usepackage{verbatim}
\usepackage[all,cmtip]{xy}
\usepackage[top=3.54cm,bottom=4.54cm,left=2.94cm,right=2.94cm ]{geometry}

\renewcommand{\Bbb}{\mathbb}

\newcommand{\Td}{{\cal D}}
\newcommand{\Te}{{\cal E}}
\newcommand{\Tf}{{\cal F}}
\newcommand{\Tt}{{\cal T}}

\newcommand{\p}{\partial}

\newcommand{\nn}{\nonumber}

\newcommand{\vep}{\varepsilon}

\renewcommand{\theequation}{\arabic{section}.\arabic{equation}}
\newcommand{\be}{\begin{equation}}
\newcommand{\ee}{\end{equation}}
\newcommand{\ba}{\begin{array}}
\newcommand{\ea}{\end{array}}
\newcommand{\bea}{\begin{eqnarray}}
\newcommand{\eea}{\end{eqnarray}}
\newcommand{\beas}{\begin{eqnarray*}}
\newcommand{\eeas}{\end{eqnarray*}}

\newtheorem{remark}{Remark}[section]

 \newcommand{\bx}{{\bf x} }
\newcommand{\sinc}{{\rm sinc}}
\newcommand{\bd}{{\beta}}

\title{Uniform and optimal error estimates of an exponential wave
 integrator sine pseudospectral method for
the nonlinear Schr\"{o}dinger equation with wave operator\thanks{This
work was supported by the Singapore A*STAR SERC  PSF-Grant 1321202067 (W. Bao)
and the KI-NET at the Center for Scientific Computation \& Math Modeling (CSCAMM) in
the University of Maryland (Y. Cai).}}
\author{Weizhu Bao\thanks{Department of Mathematics and Center for Computational
Science and Engineering, National University of Singapore, Singapore
119076 ({\tt bao@math.nus.edu.sg}, URL:
http://www.math.nus.edu.sg/\~{}bao/)} \and Yongyong
Cai\thanks{Beijing Computational Science Research Center, Beijing 100084, P. R. China; and 
Department of Mathematics, National University of
Singapore, Singapore 119076 ({\tt chainrules@gmail.com})}. Current address: Center of Scientific Computation
and Mathematical Modeling (CSCAMM), University of Maryland, College Park, MD 20742, USA}

\date{}
\begin{document}

\maketitle


\begin{abstract}We propose an exponential wave integrator
sine pseudospectral (EWI-SP) method  for the nonlinear Schr\"{o}dinger
equation (NLS) with wave operator (NLSW), and carry out rigorous error analysis.
The NLSW is  NLS perturbed by the wave operator with strength described by a
 dimensionless parameter $\vep\in(0,1]$. As $\vep\to0^+$, the NLSW converges to
 the NLS and for the small perturbation, i.e. $0<\vep\ll1$, the solution of the NLSW  differs
from that of the NLS with  a function oscillating in time  with  $O(\vep^2)$-wavelength at
$O(\vep^2)$ and $O(\vep^4)$ amplitudes for ill-prepared and well-prepared initial data,
respectively. This rapid oscillation  in time brings significant difficulties in designing
and analyzing numerical methods with  error bounds uniformly in $\vep$.   In this
work, we show that the proposed EWI-SP possesses the optimal uniform error bounds at
 $O(\tau^2)$  and $O(\tau)$ in $\tau$ (time step) for well-prepared initial data and
 ill-prepared initial data, respectively, and spectral accuracy in $h$ (mesh size) for
 the both cases, in the $L^2$ and semi-$H^1$ norms. This result significantly improves the
 error bounds of the  finite difference methods  for the NLSW. Our approach
 involves a careful study of the error propagation, cut-off of the nonlinearity
 and the energy method. Numerical examples are provided to confirm our theoretical
 analysis.
\end{abstract}

\bigskip
\bigskip



\pagestyle{myheadings}
\thispagestyle{plain}
\markboth{W. Bao and Y. Cai} {Optimal uniform error estimates of EWI-SP for NLSW}

\section{Introduction}
\label{s1} \setcounter{equation}{0}

In this paper, we consider the nonlinear Schr\"{o}dinger equation with
wave operator (NLSW)   in $d$ ($d=1,2,3$) dimensions   as \cite{Baoc,Mach,Colin1,Colin2,Schoene,Tsu,
BDX,Xin}
\begin{equation}\label{eq:nls_wave_full}
\begin{cases}
i\p_t \psi(\bx,t)-\vep^2 \p_{tt}\psi(\bx,t)+\nabla^2
\psi(\bx,t)+f(|\psi(\bx,t)|^2)\psi(\bx,t)=0,& \bx\in\Bbb R^d,\,\,t>0,\\
\psi(\bx,0)=\psi_0(\bx),\quad\quad
\p_t\psi(\bx,0)=\psi_1^\vep(\bx), & \bx\in \Bbb R^d,
\end{cases}
\end{equation}
where $\psi:=\psi(\bx,t)$ is a complex function,  $\bx$ is
the spatial variable, $t$ is the time,  $0<\vep\leq1$ is a
dimensionless parameter, $f :[0,+\infty)\to \Bbb R$ is a
real-valued function, and $\nabla^2=\Delta$ is  the Laplace
operator in $d$ dimensional space. In this paper, we will consider $\psi_1^\vep$ to be $O(1)$
w.r.t $\vep$.

The NLSW has different physical applications, including
the nonrelativistic limit of the Klein-Gordon equation \cite{Mach,Schoene,Tsu},
the Langmuir wave envelope approximation in plasma\cite{Colin1,Colin2},
and the modulated planar pulse approximation of
the sine-Gordon equation for light bullets \cite{BDX,Xin}.

As proven  by \cite{Mach,Colin1,Schoene,Tsu}, when $\vep\to0^+$, the NLSW (\ref{eq:nls_wave_full})  converges to the standard
nonlinear Schr\"{o}dinger equation (NLS),
\begin{equation}\label{eq:nls_full}
\begin{cases}
i\p_t \psi(\bx,t)+\nabla^2 \psi(\bx,t)+f(|\psi(\bx,t)|^2)\psi(\bx,t)=0,&
 \bx\in \Bbb R^d,\,\,t>0,\\
\psi(\bx,0)=\psi_0(\bx),\quad\quad & \bx\in \Bbb R^d.
\end{cases}
\end{equation}

Due to the wave operator, the solution of the NLSW (\ref{eq:nls_wave_full}) differs from the solution of the NLS
(\ref{eq:nls_full}) with a function oscillating in time $t$ with $O(\vep^2)$ wavelength. Indeed, to measure the difference between
the NLSW (\ref{eq:nls_wave_full}) and the NLS (\ref{eq:nls_full}), we can  write  the initial
 velocity $\psi_1^\vep$ for the NLSW (\ref{eq:nls_wave_full}) as
\be\label{eq:ini_full}
\psi_1^\vep(\bx)=i\left(\nabla^2\psi_0(\bx)+
f(|\psi_0(\bx)|^2)\psi_0(\bx)\right)+\vep^\alpha \omega^\vep(\bx),\quad \alpha\ge0,
\ee
where $i\left(\nabla^2\psi_0(\bx)+
f(|\psi_0(\bx)|^2)\psi_0(\bx)\right)$ corresponds to the initial velocity for the NLS (\ref{eq:nls_full}). Then
we have the following asymptotic expansion
for the solution $\psi(\bx,t)$ of the NLSW (\ref{eq:nls_wave_full}) as $\vep\to0^+$
(cf. \cite{Baoc})
\bea\label{eq:asy_nlsw}
\psi(\bx,t)&=&\psi^s(\bx,t)+
\vep^2\{\text{terms  without oscillation}\}\\
&&+\vep^{2+\min\{\alpha,2\}}\Psi(\bx,t/\vep^2)+
{\text{higher order terms  with oscillation}}, \nonumber
\eea
where $\psi^s:=\psi^s(\bx,t)$ satisfies the NLS (\ref{eq:nls_full}).

Based on this asymptotic expansion, we can make
assumptions (A) and (B) (cf. section 2.2) on the NLSW.
Furthermore, from (\ref{eq:asy_nlsw}), we identify
two cases for the imposed initial data (\ref{eq:ini_full}),
i.e., the well-prepared initial data when
 $\alpha\ge2$ and the ill-prepared initial data when
$0\leq\alpha<2$.

Various numerical methods have been developed in the literatures
for  the NLS, including the time-splitting pseudospectral method
\cite{Hardin,Taha1,Robinson,Bao2,Baocc,Besse,Lubich,Neu}, finite difference method
\cite{Akrivis1,Bao0,Baocc,Chang1,Glassey}, etc. Meanwhile, there are few  numerical methods
 for the NLSW and conservative finite difference methods \cite{Wang,Guo,Colin2}
 are the most popular ones. In our recent work \cite{Baoc},
 both the conservative Crank-Nicolson finite difference scheme
 and the semi-implicit finite difference scheme  have been analyzed
 for  NLSW. We have proved  the uniform $l^2$ and semi-$H^1$ error
 estimates of these two schemes w.r.t. $\vep\in(0,1]$,  at the
 order of $O(h^2+\tau)$ for well-prepared initial data and $O(h^2+\tau^{2/3})$
 for ill-prepared initial data,  in all dimensions ($d=1,2,3$), with mesh
 size $h$ and time step $\tau$. However, the uniform convergence rates
 are not optimal in $\tau$.  Here, we propose an exponential wave
 integrator sine pseudospectral (EWI-SP) method
 with uniform spectral accuracy in space, and the method has the optimal
 uniform error bounds   in time at the order $O(\tau^2)$  for well-prepared
 initial data and $O(\tau)$ for ill-prepared initial data. EWI-SP greatly
 improves the uniform convergence rate compared to the finite difference methods,
 and the uniform error bounds are optimal for the well-prepared initial data case.
 In fact, the exponential wave integrator  has been widely
 used in solving highly oscillatory PDE \cite{Bao1} and
 ODE problems \cite{Gau,Hoch}, and the main advantage of EWI-SP is that only
 second order derivatives in time are involved, which is the key point in
 our analysis to get the optimal and uniform error bounds.
 Other techniques of the analysis include the cut-off of the nonlinearity,
 energy method, and recursion properties of the scheme.

This paper is organized as follows. In section 2, we introduce
EWI-SP method and our main results. Section 3 is devoted to the
error estimates of EWI-SP in the case of well-prepared initial data.
In section 4, ill-prepared initial data case is considered. Numerical
results are reported in section 5 to confirm the error estimates. Then
some conclusions are drawn in section 6. Throughout the paper,  $C$ will
denote a generic constant independent of $\vep$, mesh size $h$ and time
step $\tau$, and we use the notation $A\lesssim B$ to mean
that there exists a generic constant $C$ which is independent of $\vep$,
time step $\tau$ and mesh size $h$ such that $|A|\le C\,B$.

\section{Exponential-wave-integrator sine pseudospectral method and results}\label{s2}\setcounter{equation}{0}
In practical computation, the NLSW (\ref{eq:nls_wave_full})
is usually  truncated on  a bounded  interval $\Omega=(a,b)$
in 1D ($d=1$), or a bounded rectangle $\Omega=(a,b)\times(c,d)$
in 2D ($d=2$) or a bounded box $\Omega=(a,b)\times(c,d)\times(e,f)$ in
3D ($d=3$),   with zero Dirichlet boundary condition. This truncation is accurate as long as the solution
of (\ref{eq:nls_wave_full}) stays localized. For the simplicity
of notation, we only deal with the case in 1D, i.e. $d=1$ and $\Omega=(a, b)$.
Extensions to 2D and 3D are straightforward, and the error estimates
in $L^2$-norm and semi-$H^1$ norm are the same in 2D and 3D (cf. Remark \ref{rmk:ext}). In 1D, NLSW
(\ref{eq:nls_wave_full}) is truncated on the interval $\Omega=(a,b)$ as
\begin{equation}\label{eq:nls_wave}
\begin{cases}
i\p_t \psi(x,t)-\vep^2 \p_{tt}\psi(x,t)+\p_{xx} \psi(x,t)+f(|\psi(x,t)|^2)
\psi(x,t)=0,& x\in\Omega\subset \Bbb R,\,\,t>0,\\
\psi(x,0)=\psi_0(x),\quad\quad
\p_t\psi(x,0)=\psi_1^\vep(x), & x\in \Omega,\\
\psi(x,t)\big|_{\partial\Omega}=0, &t>0.
\end{cases}
\end{equation}
As $\vep\to 0^+$, the  solution of equation (\ref{eq:nls_wave})
will  converge to the solution of the corresponding  NLS
\cite{Colin1,Schoene,Tsu,Baoc}
\begin{equation}\label{eq:nls}
\begin{cases}
i\p_t \psi^s(x,t)+\p_{xx} \psi^s(x,t)+f(|\psi^s(x,t)|^2)\psi^s(x,t)=0,&
x\in\Omega\subset \Bbb R,\,\,t>0,\\
\psi^s(x,0)=\psi_0(x),& x\in\Omega,\\
 \psi^s(x,t)\big|_{\p\Omega}=0,& t>0.
\end{cases}
\end{equation}

In the spirit of (\ref{eq:ini_full}), we assume the initial data satisfy the condition
\be\label{eq:ini}
\psi_1^\vep(x)=\psi_1(x)+\vep^\alpha \omega^\vep(x),\quad x\in \Omega,
\ee
where $\psi_1(x)=i\left(\p_{xx} \psi_0(x)+f(|\psi_0(x)|^2)\psi_0(x)\right)$
is the value of $\p_t \psi^s(x,t)|_{t=0}$ with $\psi^s(x,t)$
being the solution of the limiting NLS (\ref{eq:nls}), $\omega^\vep$ is
uniformly bounded in $H_0^1\cap H^2$ (w.r.t. $\vep$) and satisfies
$\liminf_{\vep\to0^+}\|\omega^\vep\|_{H^2}>0$ and $\alpha\ge0$ is
a parameter describing the compatibility of the initial data with
respect to the limiting NLS (\ref{eq:nls}).

\subsection{EWI-SP method}\label{subsec:ewi}
For simplicity of notations, we define  $\alpha^\ast$ as
 \be\label{ineq}
 \alpha^\ast=\min\{\alpha,2\}.
 \ee
 We also denote the nonlinear term $f(|z|^2)z$ as
 \be
 \Tf(z)=f(|z|^2)z,\quad \forall z\in\Bbb C.
 \ee
Choose time step
$\tau:=\Delta t$ and denote time steps as $t_n:=n\,\tau$ for
$n=0,1,2,\ldots$; choose mesh size $\Delta x:=\frac{b-a}{M}$
with $M$  being a positive integer and
denote $h:=\Delta x$ and  grid points
as
\[x_j:=a+j\,\Delta x,\quad j=0,1,\ldots, M.\]
Define the index sets
 \begin{eqnarray*}
 &&{\cal T}_{M}=\{j\ |\ j=1,2,\ldots,M-1\},
 \qquad{\cal T}_{M}^0=\{j\ |\ j=0,1,2\ldots,M\}.
 \end{eqnarray*}
and denote
\be\label{X_M}\begin{split}
&X_M=\text{span}\left\{\Phi_l(x)=\sin\left(\mu_l(x-a)\right),
\quad \mu_l=\frac{\pi l}{b-a},\quad x\in\Omega,\,l\in{\cal T}_M\right\},\\
&Y_M=\left\{v=(v_0,v_1,\ldots,v_{M})^T\in \Bbb C^{M+1}\bigg|\,v_0=v_{M}=0\right\}.
\end{split}
\ee
We define the finite difference operators as usual, for
$\psi(x_j,t_n)$ ($j\in\Tt_M^0$, $n\ge1$),
\begin{equation}
\delta_x^+\psi(x_j,t_n)=\frac{1}{h}(\psi(x_{j+1},t_n)-
\psi(x_j,t_n)),\quad
\delta_t^-\psi(x_j,t_n)=\frac{1}{\tau}(\psi(x_j,t_{n})
-\psi(x_j,t_{n-1})).
\end{equation}

 The {\sl  exponential wave integrator  sine spectral method}
 of NLSW (\ref{eq:nls_wave}) is to apply sine spectral
 method for spatial discretization and exponential wave integrator
 for time. For any function $\psi(x)\in L^2(\Omega)$, $\phi(x)\in C_0(\bar{\Omega})$,
 and vector $\phi=(\phi_0,\phi_1,\ldots,\phi_{M})^T\in Y_M$, let
 $P_M:L^2\left(\Omega\right)\to X_M$ be the standard $L^2$ projection
 onto $X_M$ and $I_M:C_0(\overline{\Omega})\to X_M$  and $I_M: Y_M\to X_M$
 be the standard sine interpolation operator as
\be\label{interpolation}
\left(P_M\psi\right)(x)=\sum\limits_{l=1}^{M-1}\hat{\psi}_l
\sin\left(\mu_l(x-a)\right),\qquad \left(I_M\phi\right)(x)
=\sum\limits_{l=1}^{M-1}\tilde{\phi}_l
\sin\left(\mu_l(x-a)\right),\quad x\in\bar{\Omega}=[a,b],
\ee
and the coefficients are given by
\be\label{eq:integ}
\hat{\psi}_l=\frac{2}{b-a}\int_a^b\psi(x)\sin\left(\mu_l(x-a)\right)\,dx,\qquad \tilde{\phi}_l=\frac{2}{M}\sum\limits_{j=1}^{M-1}\phi_j\sin\left(jl\pi/M\right),\quad l\in{\cal T}_M,
\ee
where $\phi_j=\phi(x_j)$ when $\phi$ is a function instead of a vector.

The sine spectral discretization for equation (\ref{eq:nls_wave}) becomes:

\noindent Find
\be\label{eq:sinsp}
\psi_M:=\psi_M(x,t)=\sum\limits_{l=1}^{M-1}
\widehat{\psi}_l(t)\sin(\mu_l(x-a)),\quad x\in \Omega,\, t\ge0,
\ee
such that
\be\label{wif}
(i\p_t-\vep^2\p_{tt})\psi_{M}(x,t)+\Delta\psi_M(x,t)+
P_M\left(f(|\psi_M|^2)\psi_M\right)=0,\quad x\in\Omega,\,t>0.
 \ee

 Substituting (\ref{eq:sinsp}) into (\ref{wif}) and making use
 of the fact that $\Phi_l(x)=\sin(\mu_l(x-a))$ (\ref{X_M}) are orthogonal to each other,
 we have
 \be\label{eq:psim1}
 -\vep^2\frac{d^2}{dt^2}\widehat{\psi}_l(t)+i\frac{d}{dt}\widehat{\psi}_l(t)
 -|\mu_l|^2\widehat{\psi}_l(t)+\left(\widehat{\Tf(\psi_M)}\right)_l
 =0,\quad l\in{\cal T}_M,\, t>0.
 \ee
Then, a numerical method can be designed by properly treating the above second order ODEs (\ref{eq:psim1}) \cite{Bao1,Gau,Hoch}, e.g.
solving (\ref{eq:psim1}) via variation of constant formula and approximating the convolution term in a good manner.
Now, let us state our approach in detail.
For each $l\in{\cal T}_M$, around time $t_n=n\tau$ ($n\ge0$),
we reformulate the above ODE as
\be\label{ODE}
-\vep^2\frac{d^2}{dt^2}\widehat{\psi}_l(t_n+s)+i\frac{d}{dt}\widehat{\psi}_l(t_n+s)
 -|\mu_l|^2\widehat{\psi}_l(t_n+s)+f_l^n(s)=0,\quad s\in\Bbb R,
\ee
where
\be\label{eq:fl}
f_l^n(s)=\left(\widehat{\Tf(\psi_M)}\right)_l(t_n+s).
\ee
For $l\in{\cal T}_M$ and $s\in\Bbb R$, we denote
\be\label{keldef}
\begin{split}
&\beta_l^{+}=\frac{1+\sqrt{1+4\vep^2|\mu_l|^2}}{2\vep^2}=O\left(\frac{1}{\vep^2}\right),\qquad \beta_{l}^{-}=\frac{1-\sqrt{1+4\vep^2|\mu_l|^2}}{2\vep^2}=
\frac{-2|\mu_l|^2}{1+\sqrt{1+4\vep^2|\mu_l|^2}}O(1),\\
&\bd_l=\beta_l^+-\beta_l^-=\frac{\sqrt{1+4\vep^2|\mu_l|^2}}{\vep^2}=O\left(\frac{1}{\vep^2}\right), \qquad
\kappa_l(s)=e^{is\beta_l^+}-e^{is\beta_l^-}=O(1).
\end{split}
\ee
Solving the above second order ODE (\ref{ODE}), e.g., using the
variation-of-constant formula, the solution can be
written as follows, for any $s\in\Bbb R$,
\be\label{eq:sol}
\widehat{\psi}_l(t_n+s)=\gamma_l^ne^{i\beta_{l}^{+}s}+\nu_l^ne^{i\beta_{l}^{-}s}
-\frac{i}{\vep^2\bd_l}
\int_0^sf_l^n(s_1)\kappa_l(s-s_1)\,ds_1,
\ee
with
\be\label{eq:sol2}
\gamma_l^n=-\frac{\beta_{l}^{-}\widehat{\psi}_l(t_n)+
i\p_t\widehat{\psi}_l(t_n)}{\bd_l},
\quad \nu_l^n=\frac{\beta_{l}^{+}\widehat{\psi}_l(t_n)+
i\p_t\widehat{\psi}_l(t_n)}{\bd_l}.
\ee
Using the formula (\ref{eq:sol}) and (\ref{eq:sol2}),
we are going to determine the suitable approximations of
$\widehat{\psi}_l$ at different time steps. For $n=0$,
let $s=\tau$ in (\ref{eq:sol}), then we get
\be\label{eq:n=1}
\widehat{\psi}_l(\tau)=\gamma_l^0e^{i\beta_{l}^{+}\tau}
+\nu_l^0e^{i\beta_{l}^{-}\tau}
-\frac{i}{\vep^2\bd_l}
\int_0^\tau f_l^0(s)\kappa_l(\tau-s)\,ds,
\ee

For $n\ge1$, choosing $s=\pm \tau$ in (\ref{eq:sol}),
and eliminating the $\p_t\widehat{\psi}_l(t_n)$ term,
we have
\begin{equation}
\begin{split}\label{eq:n>1}
\widehat{\psi}_l(t_{n+1})=&-e^{\frac{i\tau}{\vep^2}}
\widehat{\psi}_l(t_{n-1})+2\,e^{\frac{i\tau}{2\vep^2}}\cos(\tau\bd_l/2)
\widehat{\psi}_l(t_n)-\frac{i}{\vep^2\bd_l}
\int_0^\tau f_l^n(s)\kappa_l(\tau-s)\,ds\\&
+i\frac{e^{i\tau/\vep^2}}{\vep^2\bd_l}
\int_0^\tau f_l^{n-1}(s)\kappa_l(-s)\,ds.
\end{split}
\end{equation}
 For convenience, we recall the definition of `sinc' function as
\be
\sinc(s)=\frac{\sin(s)}{s},\quad \text{for}\quad s\neq 0, \,
\text{and}\quad \sinc(0)=1,
\ee
and introduce the following notations for $l\in{\cal T}_M$,
\begin{equation}
\begin{split}
&\sigma_l^+(s)=e^{is\beta_l^+/2}\sinc(s\beta_l^+/2),
\quad\sigma_l^-(s)=e^{is\beta_l^-/2}\sinc(s\beta_l^-/2),\quad s\in\Bbb R.
\end{split}
\end{equation}
We approximate the  integrals in (\ref{eq:n=1}) and (\ref{eq:n>1}) as follows:
\begin{equation}\label{eq:approxn=1}
\begin{split}
\int_0^\tau f_l^0(s)\kappa_l(\tau-s)\,ds\approx &\,\int_0^\tau \left(f_l^0(0)+s\,\p_tf_l^0(0)\right)\kappa(\tau-s)ds\\
=&\tau f_l^0(0)
\left(\sigma_l^+(\tau)-\sigma_l^-(\tau)\right)
+i\tau\p_tf_l^0(0)
\left(\frac{1-\sigma_l^+(\tau)}{\beta^+_l}
-\frac{1-\sigma_l^-(\tau)}{\beta^-_l}\right),\\
\int_0^\tau f_l^0(s)\kappa_l(-s)\,ds\approx&\,\tau f_l^0(0)
\left(\overline{\sigma_l^+(\tau)}-\overline{\sigma_l^-(\tau)}\right)
+i\tau\p_tf_l^0(0)
\left(\frac{1-\sigma_l^+(\tau)}{\beta^+_le^{i\tau\beta_l^+}}
-\frac{1-\sigma_l^-(\tau)}{\beta^-_le^{i\tau\beta_l^-}}\right),
\end{split}
\end{equation}
where $\bar{c}$ denotes the complex conjugate of $c$ and $\p_tf_l^0(0)$ can be computed exactly since
$\p_t\psi(t=0)$ is known using the initial data.
For $n\ge1$, we use the similar approximation as above
and approximate $\p_tf_l^n(0)$ by finite difference $\delta_t^-f_l^n(0)$,
\begin{align*}
\int_0^\tau f_l^n(s)\kappa_l(\tau-s)\,ds\approx & \tau f_l^n(0)\left(\sigma_l^+(\tau)-\sigma_l^-(\tau)\right)
+i\tau\delta_t^-f_l^n(0)\left(\frac{1-\sigma_l^+(\tau)}{\beta^+_l}
-\frac{1-\sigma_l^-(\tau)}{\beta^-_l}\right),\\
\int_0^\tau f_l^n(s)\kappa_l(-s)\,ds
\approx&\tau f_l^n(0)\left(\overline{\sigma_l^+(\tau)}-
\overline{\sigma_l^-(\tau)}\right)
+i\tau\delta_t^-f_l^n(0)\left(\frac{1-\sigma_l^+(\tau)}{\beta^+_le^{i\tau\beta_l^+}}
-\frac{1-\sigma_l^-(\tau)}{\beta^-_le^{i\tau\beta_l^-}}\right).
\end{align*}
For convenience of notations, we  denote $\Td(\phi(x,t_n))$ for $\phi(x,t)$ as
\begin{equation}\label{Tddef}
\Td(\phi(0))=\frac{d}{dt}\Tf(\phi(t))|_{t=0},
\quad \Td(\phi(t_n))=\delta_t^-\Tf(\phi(t_n)),\quad n\ge1.
\end{equation}

Thus,  sine spectral method for solving (\ref{eq:nls_wave})
can be derived as follows. For $n\ge0$, let $\psi_M^{n}(x)$
be the  approximations of $\psi_M(x,t_n)$.
Choose $\psi_M^0=P_M(\psi_0)$, then $\psi_M^{n+1}(x)$
for $n\ge0$ is given by
\begin{equation}\label{scheme1:1}
\psi_M^{n+1}(x)=\sum\limits_{l=1}^{M-1}
(\widehat{\psi_M^{n+1}})_l\sin(\mu_l(x-a)),\quad x\in\Omega,\quad n=0,1,\ldots,
\end{equation}
where the sine transform coefficients can be obtained
via the following formula
\begin{equation}\label{scheme1:2}
\begin{split}
(\widehat{\psi^{1}_M})_l=&c_l^0(\widehat{\psi_M^{0}})_l+
d_l^0(\widehat{\psi_1^\vep})_l
+p_l(\widehat{\Tf(\psi_M^0)})_l
+q_l(\widehat{\Td(\psi^0_M)})_l,\quad\text{and for}\;  n\ge1\; \text{as}\\
(\widehat{\psi^{n+1}_M})_l=&c_l(\widehat{\psi_M^{n-1}})_l
+d_l(\widehat{\psi_M^{n}})_l+p_l(\widehat{\Tf(\psi_M^n)})_l
+q_l(\widehat{\Td(\psi^n_M)})_l
-p_l^{\ast}(\widehat{\Tf(\psi_M^{n-1})})_l
-q_l^{\ast}(\widehat{\Td(\psi^{n-1}_M)})_l,
\end{split}\end{equation}
where $\Td(\psi_M^{n})$ ($n\ge0$) is given by (\ref{Tddef}) as
\begin{equation}\label{scheme1:3}
\begin{split}
\Td(\psi_M^0)=G(\psi_M^0)P_M(\psi_1^\vep)+
H(\psi^0_M)\overline{P_M(\psi_1^\vep)}, \quad \Td(\psi_M^n)=
\delta_t^-\Tf(\psi_M^n),\quad n\ge1,
\end{split}
\end{equation}
with
\begin{equation}\label{coef1}
\begin{split}
&G(z)=f(|z|^2)+f^\prime(|z|^2)\cdot|z|^2,\quad H(z)=
f^\prime(|z|^2)z^2,\qquad\forall z\in\Bbb C,\\
&c_l^0=\frac{\beta_l^+e^{i\tau\beta_l^-}-\beta_l^-
e^{i\tau\beta_l^+}}{\bd_l},\quad
d_l^0=-i\tau e^{i\frac{\tau}{2\vep^2}}
\sinc(\tau\frac{\bd_l}{2}),\quad p_l=\frac{-i\tau}{\vep^2\bd_l}
(\sigma_l^+(\tau)-\sigma_l^-(\tau)),\\
&q_l=\frac{\tau}{\vep^2\bd_l}\left(
\frac{1-\sigma_l^+(\tau)}{\beta^+_l}
-\frac{1-\sigma_l^-(\tau)}{\beta^-_l}\right),\quad
p_l^{\ast}=\frac{-i\tau e^{\frac{i\tau}{\vep^2}}}{\vep^2\bd_l}(\overline{\sigma_l^+(\tau)}
-\overline{\sigma_l^-(\tau)}),\\
&
q_l^{\ast}=\frac{\tau e^{\frac{i\tau}{\vep^2}}}{\vep^2\bd_l}\left(\frac{1-\sigma_l^+(\tau)}{e^{i\tau\beta_l^+}\beta^+_l}
-\frac{1-\sigma_l^-(\tau)}{e^{i\tau\beta_l^-}\beta^-_l}\right),
\quad c_l=-e^{\frac{i\tau}{\vep^2}},\quad
d_l=2e^{i\frac{\tau}{2\vep^2}}\cos(\tau\frac{\bd_l}{2}).
\end{split}
\end{equation}
We note that $|c_l|,|d_l|,|c_l^0|\lesssim 1$, $|d_l^0|,|q_l|, |q_l^{\ast}|\lesssim \tau^2$
and $|p_l|,|p_l^{\ast}| \lesssim \tau$ ($l\in\Tt_M$).

In practice, the above approach is not suitable due to
the difficulty of computing the
sine transform coefficients in (\ref{scheme1:2}) through
the integrals given in (\ref{eq:integ}). Here, we present an efficient
implementation by choosing $\psi_M^0(x)$ as the interpolation
of $\psi_0(x)$ on the grid points $\{x_j,\,j\in\Tt_M\}$, i.e.,
$\psi_M(x,0)=I_M(\psi_0)(x)$, and
approximating the integrals in (\ref{eq:integ}) and
(\ref{scheme1:2}) by a quadrature rule on the grid points.

Now, we state the {\it exponential wave integrator
sine pseudospectral} (EWI-SP) method. Let $\psi_j^n$ ($n\ge1$)
be the  approximation of $\psi(x_j,t_n)$ with $\psi_0^{n}=\psi_{M}^n=0$,
and we denote $\psi_j^0=\psi_0(x_j)$, $\omega_j^\vep=\omega^\vep(x_j)$
($j\in\Tt_M$),  $\psi^n=(\psi_0^n,\psi_1^n,\ldots,\psi_{M}^n)^T\in Y_M$.
Thus,  the numerical approximation $\psi^{n+1}\in Y_M$ at time $t_{n+1}$
($n=0,1,\ldots,$) can be computed as
\begin{equation}\label{scheme2:1}
\psi^{n+1}_j=\sum\limits_{l=1}^{M-1}
(\widetilde{\psi^{n+1}})_l\sin(jl\pi/M),\quad j\in\Tt_M,\quad n=0,1,\ldots,
\end{equation}
where the  coefficients are obtained through the following formula,
\begin{equation}\label{scheme2:2}
\begin{split}
(\widetilde{\psi^{1}})_l=&c_l^0(\widetilde{\psi_0})_l
+ d_l^0(\widetilde{\psi_1^\vep})_l
+p_l(\widetilde{\Tf(\psi_0)})_l
+q_l(\widetilde{\Td(\psi_0)})_l,\quad\text{and for} \; n\ge1 \;\text{as}\\
(\widetilde{\psi^{n+1}})_l=&c_l^1(\widetilde{\psi^{n-1}})_l+ d_l^1(\widetilde{\psi^{n}})_l
+p_l(\widetilde{\Tf(\psi^n)})_l
+q_l(\widetilde{\Td(\psi^n)})_l
-p_l^{\ast}(\widetilde{\Tf(\psi^{n-1})})_l-q_l^{\ast}
(\widetilde{\Td(\psi^{n-1})})_l,
\end{split}\end{equation}
and $\Td(\psi_0)$ is  defined by (\ref{Tddef})   as
\begin{equation}\label{scheme2:3}
\begin{split}
\Td(\psi_0)=G(\psi_0)\psi_1^\vep+
H(\psi_0)\overline{\psi_1^\vep}, \quad \Td(\psi^n)=
\delta_t^-\Tf(\psi^n),\quad n\ge1,
\end{split}
\end{equation}
with  $c_l^0$, $d_l^0$, $c_l$, $d_l$, $p_l$, $q_l$,  $p_l^{\ast}$
and $q_l^{\ast}$ given in (\ref{coef1}).
In computation,
sometimes $\psi_1$ (\ref{eq:ini}) is not explicitly known, we can
replace $(\widetilde{\psi_1^\vep})_l$ ($l\in\Tt_M$) in (\ref{scheme2:2})
by the approximation below, and the main results remain the same,
\be\label{scheme2:4}
(\widetilde{\psi_1^\vep})_l\approx-i(\mu_l)^2
(\widetilde{\psi_0})_l+i(\widetilde{\Tf(\psi_0)})_l+\vep^\alpha
(\widetilde{\omega^\vep})_l.
\ee
The EWI-SP (\ref{scheme2:1})-(\ref{scheme2:3}) is explicit,
and can be solved
efficiently by fast sine transform. The memory cost is $O(M)$ and
the computational cost per time step is $O(M\log M)$. Actually, the scheme
has a very good recursion
property in phase space (cf. (\ref{eq:induc1}) and
Lemma \ref{lem:localerr}). The observation is the following:
from the equation (\ref{eq:nls_wave}), we can see that there are
two characteristics corresponding to
the operators $U_{+}(t)=e^{i\frac{1+\sqrt{1+4\vep^2(-\nabla^2)}}{2\vep^2}}$
and $U_{-}(t)=e^{it\frac{2\nabla^2}{1+\sqrt{1+4\vep^2(-\nabla^2)}}}$.
As $\vep\to 0^+$, the first characteristic
component $U_+$ will vanish while oscillating in time, and the second
characteristic component $U_-$ will converge to the Schr\"odinger
operator. That is to say, the solution behaves differently
along the two characteristics.
In our scheme EWI-SP, these characteristics are
treated separately through the approximations of
the integrals in (\ref{eq:n=1}) and (\ref{eq:n>1}),
unlike  those wave-type equations (without $i\p_t$ term)
\cite{Bao1} and second order ODEs \cite{Gau,Hoch}.

The pseudospectral method EWI-SP (\ref{scheme2:1})-(\ref{scheme2:2}) is a full
discretization and the spectral method (\ref{scheme1:1})-(\ref{scheme1:2})
is a semi-discretization. For simplicity, we will  prove
the error estimates for the full discretization and omit the analysis for the
semi-discretization  which can be done in the same spirit.

\begin{remark}If we consider the NLSW (\ref{eq:nls_wave}) with
periodic boundary condition or homogenous Neumann   boundary condition,
similar Fourier pseudospectral method or cosine pseudospectral method
can be easily designed as above, and the main results in this paper
remain valid in both cases.
\end{remark}

\subsection{Main results} In order to state our main results, we
introduce the convenient Sobolev spaces. Let
$\phi(x)\in H^m(\Omega)\cap H_0^1(\Omega)$
be represented in sine series as
\be
\phi(x)=\sum\limits_{l=1}^{+\infty}\widehat{\phi}_l
\Phi_l(x),\qquad x\in\Omega=(a,b),
\ee
with $\widehat{\phi}_l$ and $\Phi_l(x)$ given in (\ref{eq:integ}) and (\ref{X_M}),
respectively. Define the subspace of $H^m\cap H_0^1$ as  $H_s^m(\Omega)=\{\phi\in H^m(\Omega)| \p_x^{2k}
\phi(a)=\p_x^{2k}\phi(b)=0,\; 0\leq 2k<m\}$ (the boundary values are understood in the trace sense)
 equipped
with the norm
\be
\|\phi\|_{H_s^m(\Omega)}=\left(\sum\limits_{l=1}^\infty
\mu_l^{2m}|\widehat{\phi}_l|^2\right)^{\frac12},
\ee
which is equivalent to the $H^m$ norm in this subspace.
 In the remaining
part of this paper, we will omit $\Omega$ when the norm is only
taken with respect to the spatial variables.

For $u=(u_0,u_1,\ldots,u_{M})^T\in Y_M$, we define the
discrete $l^2$, semi-$H^1$ and $l^\infty$ norms as
\be\label{norm}
\|u\|_{l^2}^2=h\sum\limits_{j=1}^{M-1}|u_j|^2,\quad \|\delta_x^+u\|_{l^2}^2=h\sum\limits_{j=0}^{M-1}|\delta_x^+u_j|^2,
\quad \|u\|_{l^\infty}=\max\limits_{
j\in\Tt_M}|u_j|.
\ee

According to the known results in \cite{Colin1,Mach,Schoene,Tsu}
and the asymptotic expansion in section 1,  we make the following
assumptions on NLSW (\ref{eq:nls_wave}), i.e.
assumptions on the initial data (\ref{eq:ini}) for (\ref{eq:nls_wave}),
\begin{equation*}
 \text{(A)}\qquad1\lesssim
 \|\omega^\vep(\cdot)\|_{H^{m_0}_{s}}
 \lesssim 1,\quad \psi_0\in H^{m_0+2}_{s}
 \qquad \text{for some}\quad m_0\ge2;
 \qquad\qquad\qquad\quad
   \end{equation*}
and assumptions on the solution $\psi(t):=\psi(\cdot,t)$ of the NLSW
(\ref{eq:nls_wave})
and the solution $\psi^s(\cdot,t)$ of the NLS (\ref{eq:nls}),  i.e.,
let $0<T<T_{\rm max}$ with
$T_{\rm max}$ being the maximal common existing time,
\beas \text{(B) }\;\;\quad
&&\|\psi\|_{L^\infty([0,T];L^\infty\cap H^{m_0}\cap H_0^1)}+
\|\p_t\psi\|_{L^\infty([0,T];H^1)}
+\|\psi^s\|_{L^\infty([0,T];L^\infty\cap H^{m_0}\cap H_0^1)}
\lesssim 1,\quad m_0\ge2,\\
\text{and }&&\|\partial_{tt}\psi\|_{L^\infty([0,T];H^1)}
\lesssim \frac{1}{\vep^{2-\alpha^\ast}}, \quad
\|\partial_{t}^k\psi^s\|_{L^\infty([0,T];H^1)}\lesssim 1, \quad \|\psi-\psi^s\|_{L^\infty([0,T];H^1)}\lesssim\vep^2,\quad k=1,2.
\eeas
Under assumption (A), we can verify that
$\psi_1,\psi_1^\vep\in H_s^{m_0}$ (\ref{eq:ini}) for
function $f(\cdot)$ smooth enough. Under assumption (B),
in view of the equations (\ref{eq:nls_wave}) and (\ref{eq:nls}),
and Lemma \ref{lem:reg}, it is easy to show that
$\psi,\psi^s\in L^\infty([0,T];H^{m_0}_s)$.

Denote the constant $M_1$ as
\be\label{M1}
M_1=\max\left\{\sup_{\vep\in(0,1]}\|\psi(x,t)\|_{L^\infty([0,T]\times\Omega)},\,
\|\psi^s(x,t)\|_{L^\infty([0,T]\times\Omega)}\right\}.
\ee
We can prove the following error estimates for  EWI-SP (\ref{scheme2:1})-(\ref{scheme2:3}).

\begin{theorem}\label{thm:main}(Well-prepared initial data)
Let $\psi^n\in Y_M$ and $\psi^n_I(x)=I_M(\psi^n)$ ($n\ge0$)
be the numerical approximations obtained from EWI-SP (\ref{scheme2:1})-(\ref{scheme2:3}).
Assume $f(s)\in C^{k}([0,+\infty))$ ($k\ge3$), under  assumptions {\rm (A)} and
{\rm (B)},
there exist constants $0<\tau_0,\frac{h_0}{\pi}\leq1$ independent of $\vep$, if $0<h\leq h_0$
and $0<\tau\leq\tau_0$, we have for $\alpha\ge2$, i.e.
the well-prepared initial data case,
\begin{equation}\label{eq:mainres}\begin{split}
&\|\psi(\cdot,t_n)-\psi^n_I(\cdot)\|_{L^2}\lesssim h^{m}+\tau^2,
\qquad \|\psi^n\|_{l^\infty}\leq M_1+1,\\ &\|\nabla (\psi(\cdot,t_n)
-\psi^n_I(\cdot))\|_{L^2}\lesssim h^{m-1}+\tau^2,\qquad 0\leq n\leq\frac{T}{\tau},
\end{split}\end{equation}
where $m=\min\{m_0,k\}$ and $M_1$ is defined in (\ref{M1}).
\end{theorem}

Similarly, we have for the ill-prepared initial data case.
\begin{theorem}\label{thm:main2}(Ill-prepared initial data)
Under the same condition of Theorem \ref{thm:main}, there exist constants
$0<\tau_0,\frac{h_0}{\pi}\leq1$ independent of $\vep$, if $0<h\leq h_0$
and $0<\tau\leq\tau_0$, we have
for $\alpha\in[0,2)$,  i.e. the ill-prepared initial data case,
\begin{align}
&\|\psi(\cdot,t_n)-\psi^n_I(\cdot)\|_{L^2}\lesssim h^{m}
+\frac{\tau^2}{\vep^{2-\alpha}},
\quad \|\nabla (\psi(\cdot,t_n)-\psi^n_I(\cdot))\|_{L^2}\lesssim h^{m-1}+\frac{\tau^2}{\vep^{2-\alpha}},\label{ill-res1}\\
&
\|\psi(\cdot,t_n)-\psi^n_I(\cdot)\|_{L^2}\lesssim h^{m}+\tau^2+\vep^2,\quad\|\nabla (\psi(\cdot,t_n)-\psi^n_I(\cdot))\|_{L^2}\lesssim h^{m-1}+\tau^2+\vep^2, \label{ill-res2}\\
 & \|\psi^n\|_{l^\infty}\leq M_1+1,\quad 0\leq n\leq\frac{T}{\tau},\label{ill-res3}
\end{align}
where $m=\min\{m_0,k\}$ and $M_1$ is defined in (\ref{M1}).
Thus, by taking the minimum of $\vep^2$ and $\frac{\tau^2}{\vep^{2-\alpha}}$ for $0<\vep\leq1$,
we could obtain uniform error bounds as
\begin{equation*}
\begin{split}
\|\psi(\cdot,t_n)-\psi^n_I(\cdot)\|_{L^2}\leq h^{m}+\tau^{4/(4-\alpha)},\quad
\|\nabla[\psi(\cdot,t_n)-\psi^n_I(\cdot)]\|_{L^2}\lesssim h^{m-1}+\tau^{4/(4-\alpha)},\quad 0\leq n\leq\frac{T}{\tau}.
\end{split}
\end{equation*}
\end{theorem}

\begin{remark}The main results can be extended directly to the system of $N$ ($N\ge2$) coupled nonlinear Schr\"odinger equations with wave operator (NLSW) as
\begin{equation*}
i\p_t \psi_k-\vep^2 \p_{tt}\psi_k(x,t)+\nabla^2
\psi_k(x,t)+f_k(|\psi_1|^2,\ldots,|\psi_N|^2)\psi_k(x,t)=0,\quad k=1,\ldots,N,
\end{equation*}
where the initial data and boundary conditions for each $\psi_k$ are similarly
given as (\ref{eq:nls_wave}) and (\ref{eq:ini}), and the  functions $f_k:[0,\infty)^N\to \Bbb R$ ($k=1,\ldots,N$) belong to $C^k$ ($k\ge3$) function class.
\end{remark}

Since the exact solution behaves very differently for the well-prepared initial data ($\alpha\ge2$) and the ill-prepared initial data ($0\leq\alpha<2$) (see (\ref{eq:asy_nlsw}) and \cite{Baoc}), we treat these two cases separately. For the  well-prepared initial data, in order to prove Theorem \ref{thm:main}, we will compare the EWI-SP approximation (\ref{scheme2:1})-(\ref{scheme2:4}) with the $L^2$ projection of the exact solution.
Following the steps in  section \ref{subsec:ewi}, where we construct the EWI-SP (\ref{scheme2:1})-(\ref{scheme2:4}), it is easy to find that the  errors introduced only come
from the trigonometric interpolation (spectral accurate in space) and Gautschi type quadrature \cite{Gau,Hoch} for the integrals in (\ref{eq:sol}) (second order accurate in time).
This observation gives Lemma \ref{lem:localerr} for the local truncation error, which is of spectral order in space and second order in time. In addition, because we use a Gautschi type quadrature,  the scheme (\ref{scheme2:1})-(\ref{scheme2:4}) posses a very nice recursive property which leads to the  nice recursion formula for the error (\ref{eq:induc1}).
 These are the key points in our analysis. In the proof, we also use the cutoff technique for the nonlinear term, discrete Sobolev inequality  and the  energy method.

For the ill-prepared initial data, in order to prove Theorem \ref{thm:main2} with two different estimates, we adopt a  strategy similar to the finite difference method \cite{Baoc}  (cf. diagram (\ref{chart})). The idea for establishing the estimates is analogous to that of Theorem \ref{thm:main}. The same analysis works for 2D and 3D cases, and we discuss  this
in  Remark \ref{rmk:ext}

\section{Convergence in the well-prepared initial data case}
Before the analysis of convergence, we first review and introduce some lemmas for the projection operator
$P_M$ and interpolation operator $I_M$ \cite{Shen,Kreiss}.
\begin{lemma}\label{lem:l2}($L^2$ projection) Let $\phi(x)\in H_s^{m}(\Omega)$ ($m\ge2$), then
\be
\|\phi(\cdot)-P_M(\phi)(\cdot)\|_{L^2}\lesssim \pi^{-m}h^{m},\quad \|\nabla[\phi(\cdot)-P_M(\phi)(\cdot)]\|_{L^2}\lesssim
\pi^{1-m}h^{m-1}.
\ee
\end{lemma}
For the sine trigonometric interpolation $I_M$, we have
similar results.
\begin{lemma}\label{lem:sine}(Sine interpolation)
Let $\phi(x)\in H_0^{1}(\Omega)$ and $\phi=(\phi_0,\phi_1,\ldots,\phi_M)^T$ with $\phi_j=\phi(x_j)$ ($j\in\Tt_M^0$), we have
\begin{equation*}
\|I_M(\phi)(\cdot)\|_{L^2}\leq\|\phi(\cdot)\|_{L^2}+h\|\nabla\phi(\cdot)\|_{L^2},
\quad\|\delta_x^+\phi\|_{l^2}\leq\|\nabla I_M(\phi)(\cdot)\|_{L^2}\leq \frac{\pi}{2} \|\delta_x^+\phi\|_{l^2}.
\end{equation*}
$I_M$ and $P_M$ are identity transforms on $X_M$, and if $\phi(x)\in X_M$,
\be\|\nabla\phi(\cdot)\|_{L^2}\lesssim h^{-1}\|\phi(\cdot)\|_{L^2}.\ee
\end{lemma}
\begin{proof}The lemma can be proved analogous to  \cite{Shen,Kreiss}.
We prove the first statement here for reader's convenience.
By definition of $I_M$, using Cauchy inequality, it is
easy to check that
\begin{equation}
\begin{split}
\|I_M(\phi)(\cdot)\|_{L^2}^2=&h\sum\limits_{j=1}^{M-1}|\phi(x_j)|^2=\|\phi(\cdot)\|_{L^2}^2
-\sum\limits_{j=0}^{M-1}\int_{x_j}^{x_{j+1}}\left(|\phi(\bx)|^2-|\phi(x_j)|^2\right)\,d\bx\\
=&\|\phi(\cdot)\|_{L^2}^2
-\sum\limits_{j=0}^{M-1}\int_{x_j}^{x_{j+1}}\int_{x_j}^{\bx}\p_x|\phi(\bx^\prime)|^2\,d\bx^\prime\,d\bx\\
\leq&\|\phi(\cdot)\|_{L^2}^2+2h\|\nabla\phi(\cdot)\|_{L^2}\|\phi(\cdot)\|_{L^2},
\end{split}
\end{equation}
which implies the first conclusion by using Cauchy inequality. For the second
conclusion, we note that
\begin{equation*}
\|\nabla I_M(\phi)(\cdot)\|_2^2=\frac{b-a}{2}\sum\limits_{l=1}^{M-1}|\mu_l|^2|\widetilde{\phi}_l|^2,\,
\quad \delta_x^+\phi(x_j)=\sum\limits_{l=0}^{M-1}\frac{2}{h}\widetilde{\phi}_l\sin(\frac{\mu_lh}{2})\cos(\mu_l (j+\frac12)h),\quad 0\leq j\leq M-1,
\end{equation*}
thus by Parseval's identity, we know
\be
\|\delta_x^+\phi\|_{l^2}^2=\frac{hM}{2}\sum\limits_{l=0}^{M-1}|\mu_l|^2|\widetilde{\phi}_l|^2
|\sinc(\frac{\mu_lh}{2})|^2.
\ee
Recalling that $0<\frac12\mu_lh\leq\frac\pi2$ ($l\in\Tt_M$),
$\sinc(\frac{\mu_lh}{2})\in[\frac{2}{\pi},1]$, we get the second conclusion.
\end{proof}

For  $I_M$ applied to the nonlinear term,
we require some regularity to achieve the desired accuracy.
\begin{lemma}\label{lem:reg} For $\phi(x)\in H_s^{m}(\Omega)$ ($m\ge2$)
and $f(\cdot)\in C^k([0,\infty))$ ($k\ge2$),  we have
\be
\Tf(\phi)=f(|\phi|^2)\phi\in H_s^{m^\ast},\quad \text{with}
\quad \|\Tf(\phi)\|_{H_s^{m^\ast}}\lesssim
\|\phi\|_{H_s^{m^\ast}},\quad m^\ast=\min\{m,k\}.
\ee
\end{lemma}
\begin{proof}This is  essentially due to the fact that
$\Tf(\phi)$ maps an odd function to an odd function.
The results can be checked by  direct computation
and we omit it here.
\end{proof}

Using Lemmas \ref{lem:l2} and \ref{lem:sine}, we find
for any $\Phi\in H_0^1$,
\be\label{eq:interr}
\|I_M(\Phi)-P_M(\Phi)\|_{L^2}= \|I_M[\Phi-P_M(\Phi)]\|_{L^2}\lesssim
\|\Phi-P_M(\Phi)\|_{L^2}+h\|\nabla(\Phi-P_M(\Phi))\|_{L^2}.
\ee

Let $\psi(x,t)$ be the  solution of the NLSW (\ref{eq:nls_wave}),
we write $\psi(t)$ in short of $\psi(x,t)$. In order to prove
the main theorem, we define the `local truncation error' $\xi^n(x)=\sum\limits_{l=1}^{M-1}(\widetilde{\xi^n})_l\Phi_l(x)\in X_M$
for  $l\in\Tt_M$ as
\begin{align}\label{xidef}
(\widetilde{\xi^{0})}_l=&c_l^0(\widehat{\psi_0})_l+ d_l^0(\widehat{\psi_1^\vep})_l
+p_l(\widetilde{\Tf(\psi_0)})_l
+q_l(\widetilde{\Td(\psi_0)})_l
-(\widehat{\psi(\tau)})_l,\nn\\
(\widetilde{\xi^n})_l=&c_l(\widehat{\psi(t_{n-1})})_l+ d_l(\widehat{\psi(t_n)})_l
+p_l(\widetilde{\Tf(\psi(t_n))})_l
+q_l(\widetilde{\Td(\psi(t_n))})_l
-p_l^{\ast}(\widetilde{\Tf(\psi_0)})_l\\&
-q_l^{\ast}(\widetilde{\Td(\psi(t_{n-1}))})
-(\widehat{\psi(t_{n+1})})_l,\qquad n\ge1.\nn
\end{align}
Then we have the following results.
\begin{lemma}\label{lem:localerr}Let $\xi^n(x)$ ($n\ge0$) be
defined as (\ref{xidef}), $f(\cdot)\in C^k([0,\infty))$ ($k\ge3$),
under assumptions {\rm (A)} and {\rm (B)}, we have the following decomposition
of $\xi^n(x)$ ($n\ge0$) (\ref{eq:xideco}),
 \begin{equation}\label{deco}\begin{split}
 &\xi^0(x)=e^{i\tau\beta_l^+}\xi^{0,+}(x)-e^{i\tau\beta_l^-}\xi^{0,-}(x),\qquad\\
 &\xi^n(x)=e^{i\tau\beta_l^+}\xi^{n,+}(x)-e^{i\tau\beta_l^-}\xi^{n,-}(x)
 -e^{\frac{i\tau}{\vep^2}}\left(\xi^{n-1,+}(x)-\xi^{n-1,-}(x)\right),\quad n\ge1,
 \end{split}
 \end{equation}
 where $\xi^{n,\pm}(x)=\sum\limits_{l=1}^{M-1}
 (\widetilde{\xi^{n,\pm}})_l\Phi_l(x)\in X_M$
 and for $\alpha\ge2$, i.e., the well-prepared initial data case, we have
\begin{equation}\label{xi:well}
\begin{split}
&\|\xi^{n,\pm}(\cdot)\|_{L^2}\lesssim \tau h^{m}+\tau^3,\qquad
\|\nabla\xi^{n,\pm}(\cdot)\|_{L^2}\lesssim \tau h^{m-1}+\tau^3,
\quad m=\min\{m_0,k\},\quad n\ge0;
\end{split}\end{equation}
for $\alpha\in[0,2)$, i.e. the ill-prepared initial data case,
we have
\begin{equation}\label{xi:ill}
\begin{split}
&\|\xi^{n,\pm}(\cdot)\|_{L^2}\lesssim \tau h^{m}+\frac{\tau^3}
{\vep^{2-\alpha}},\quad \|\nabla\xi^{n,\pm}(\cdot)\|_{L^2}\lesssim \tau h^{m-1}+\frac{\tau^3}{\vep^{2-\alpha}},\quad m=\min\{m_0,k\},\quad n\ge0.
\end{split}
\end{equation}
In addition, the coefficients $p_l$, $q_l$, $p_l^{\ast}$
and $q_l^{\ast}$ ($n\ge0$) given in (\ref{coef1}) can be split as
\begin{equation*}\begin{split}
&p_l=e^{i\tau\beta_l^+}p_l^{+}-e^{i\tau\beta_l^-}p_l^{-},\quad
q_l=e^{i\tau\beta_l^+}q_l^{+}-e^{i\tau\beta_l^-}q_l^{-},\quad
p_l^{\ast}=e^{\frac{i\tau}{\vep^2}}(p_l^{+}-p_l^{-}),\quad
q_l^{\ast}=e^{\frac{i\tau}{\vep^2}}(q_l^{+}-q_l^{-}),\\
& p_l^{+}=\frac{-i\tau}{\vep^2\bd_l}\overline{\sigma_l^+(\tau)},
 \quad p_l^{-}=\frac{-i\tau}{\vep^2\bd_l}\overline{\sigma_l^-(\tau)},\quad
q_l^{+}=\frac{\tau}{\vep^2\bd_l}\cdot\frac{1-\sigma_l^+(\tau)}
{e^{i\tau\beta_l^+}\beta^+_l},
\quad
q_l^{-}=\frac{\tau}{\vep^2\bd_l}\cdot\frac{1-\sigma_l^-(\tau)}
{e^{i\tau\beta_l^-}\beta^-_l},
\end{split}
\end{equation*}
where $|p_l^{\pm}|\lesssim \tau$, $|q_l^\pm|\lesssim \tau^2$.
\end{lemma}
\begin{proof}
Multiplying both sides of  NLSW (\ref{eq:nls_wave})
by $\Phi_l(x)=\sin(\mu_l(x-a))$ and integrating over $\Omega$,
we easily recover the equations for $(\widehat{\psi})_l(t)$,
which are exactly the same as (\ref{eq:psim1}) with $\psi_M$ being
replaced by $\psi(x,t)$.  Replacing $\psi_M$ with $\psi(x,t)$, we
use the same notations $f_l^n(s)$  as in (\ref{eq:fl}), and it
is worth  noticing that the time derivatives of  $f_l^n(s)$,
enjoy the same properties  of time derivatives for $\Tf(\psi(x,t))$.
Thus, the same representations (\ref{eq:n=1}) and (\ref{eq:n>1})
hold for $(\widehat{\psi(t_n)})_l$ with $n=1$ and $n\ge2$, respectively.

Now it is clear that the error $\xi^n(x)$   comes from
the approximations for the integrals (\ref{scheme1:2})
and trigonometric interpolation (\ref{scheme2:2}).
First, we can write the integrals in (\ref{eq:n=1})
and (\ref{eq:n>1}) as
\be
\frac{i}{\vep^2\bd_l}\int_0^\tau f_l^n(s)\kappa_l(\tau-s)\,ds=
e^{i\tau\beta_l^+}\frac{i}{\vep^2\bd_l}\int_0^\tau f_l^n(s)
e^{-i\beta_l^+s}\,ds-e^{i\tau\beta_l^-}\frac{i}{\vep^2\bd_l}
\int_0^\tau f_l^n(s)e^{-i\beta_l^-s}\,ds,
\ee
 then $\xi^n(x)$ ($n\ge0$) can be written as (\ref{deco}) with
\be\label{eq:xideco}
\xi^{n,+}(x)=\eta^{n,+}(x)+\zeta^{n,+}(x),\quad \xi^{n,-}(x)
=\eta^{n,-}(x)+\zeta^{n,-}(x),
\ee
where the interpolation error $\zeta^{n,+}(x),\,\zeta^{n,-}(x)\in X_M$ ($n\ge0$)
are determined by their sine transform coefficients for
$l\in \Tt_M$ given by
\begin{align}\label{zetadef}
(\widehat{\zeta^{0,\pm}})_l=&p_l^{\pm}[(\widetilde{\Tf(\psi_0)})_l
-(\widehat{\Tf(\psi_0)})_l]
+q_l^{\pm}[(\widetilde{G(\psi_0)\psi_1^\vep})_l
+(\widetilde{H(\psi_0)\overline{\psi_1^\vep}})_l
-(\widehat{G(\psi_0)\psi_1^\vep})_l-
(\widehat{H(\psi_0)\overline{\psi_1^\vep}})_l],\\
(\widehat{\zeta^{n,\pm}})_l=&p_l^{\pm}[(\widetilde{\Tf(\psi(t_n))})_l
-(\widehat{\Tf(\psi(t_n))})_l]
+q_l^{\pm}[(\widetilde{\delta_t^-\Tf(\psi(t_n))})_l-(\widehat{\delta_t^-\Tf(\psi(t_n))}
)_l],\quad n\ge1,\nn
\end{align}
with coefficients $p_l^\pm$ and $q_l^\pm$  given in the lemma,
and the integral approximation error $\eta^{n,+}(x),\,\eta^{n,-}(x)\in X_M$ ($n\ge0$)
are given by their sine transform coefficients as, for $l\in\Tt_M$,
\be\label{eq:etadef}
\begin{split}
(\widehat{\eta^{0,\pm}})_l=&-\frac{i}{\vep^2\bd_l}\int_0^\tau
\int_0^s\int_0^{s_1}\left(\p_{tt}f_l^0(s_2)\right)e^{-is\beta_l^\pm}ds_2ds_1ds,\quad
\text{and for}\; n\ge1,\\
(\widehat{\eta^{n,\pm}})_l=&-\frac{i}{\vep^2\bd_l}\int_0^\tau
\left(\int_0^s\int_0^{s_1}\p_{tt}f_l^n(s_2)ds_2ds_1-s\tau\int_0^1
\int_\theta^1\p_{tt}f_l^{n-1}(\theta_1\tau)
\,d\theta_1d\theta\right)e^{-is\beta_l^\pm}ds.
\end{split}
\ee
Now, to obtain estimates  (\ref{xi:well}) and (\ref{xi:ill}), we only need to estimate
$\zeta^{n,\pm}$ and $\eta^{n,\pm}$.

First, let us consider $\zeta^{n,\pm}$. For $n=0$,
recalling assumptions (A) and (B), it is easy to verify
that  $\psi_0$,  $\psi(\tau)$
belong to $H_s^{m_0}$ and $\psi_1^\vep\in H_s^{m}$ ($m=\min\{m_0,k\}$),
while $\Tf(\psi_0)$ belongs to
$H_s^{\min\{m_0,k\}}$, $G(\psi_0)\psi_1^\vep$ and
$H(\psi_0)\overline{\psi_1^\vep}$
belong to $H_s^{m^\ast}$ ($m^\ast=\min\{m,k-1\}\ge2$).
Hence, noticing (\ref{eq:interr}) and Lemma \ref{lem:l2}, we know
\be\label{eq:zeta:n=0}
\|\zeta^{0,\pm}(\cdot)\|_{L^2}\lesssim \tau(h^m+\tau h^{m^\ast})\lesssim
 \tau(h^m+\tau^2),\quad
\|\nabla\zeta^{0,\pm}(\cdot)\|_{L^2}\lesssim \tau(h^{m-1}+\tau^2).
\ee
Similarly, we can get for $n\ge1$ that
\be\label{eq:zeta:n>0}
\|\zeta^{n,\pm}(\cdot)\|_{L^2}\lesssim \tau h^{m},\quad
\|\nabla\zeta^{n,\pm}(\cdot)\|_{L^2}\lesssim \tau h^{m-1}.
\ee
Then we estimate $\eta^{n,\pm}$.
Noticing $f\in C^3$ and assumptions (A) and (B), in view of the
particular assumption $\p_{tt}\psi\sim \vep^{\alpha^\ast-2}$,
we can deduce that for $t\in[0,T]$,
\begin{equation}\label{eq:bound:le}\begin{split}
&\|\p_{tt}\Tf(\psi(t))\|_{L^2}\lesssim \vep^{\alpha^\ast-2},\quad \|\nabla[\p_{tt}\Tf(\psi(t))]\|_{L^2}\lesssim \vep^{\alpha^\ast-2}.
\end{split}
\end{equation}
The key point is the behavior of $\p_{tt}\psi$.
Substituting all the above results into (\ref{eq:etadef}),
we can easily get
\begin{equation*}\begin{split}
|(\widehat{\eta^0})_l|\lesssim  \int_0^\tau\int_0^{s}\int_0^{s_1}\bigg(|(\widehat{\p_{tt}
\Tf(\psi(s_2))})_l|\bigg)\,ds_2ds_1ds,\qquad l\in\Tt_M,
\end{split}\end{equation*}
and by applying Cauchy inequalities and Bessel inequalities
as well as (\ref{eq:bound:le}),  we  obtain
 \begin{equation*}
 \begin{split}
 \|\eta^{0,\pm}(\cdot)\|_{L^2}^2=&\frac{b-a}{2}\sum\limits_{l=1}^{M-1}|(\widehat{\eta^0})_l|^2
\lesssim\frac{(b-a)\tau^3}{2}\int_0^\tau\int_0^{s}\int_0^{s_1}\left(
\sum\limits_{l=1}^{M-1}|(\widehat{\p_{tt}\Tf(\psi(s_2))})_l|^2\right)\,ds_2ds_1ds\\
\lesssim & \tau^3\int_0^\tau\int_0^{s}\int_0^{s_1}\bigg(\|\p_{tt}\Tf(\psi(s_2))\|_{L^2}^2
\bigg)\,ds_2ds_1ds\lesssim \tau^6\vep^{2\alpha^\ast-4}.
 \end{split}
 \end{equation*}
In the same spirit, we can obtain
 \begin{equation*}
 \begin{split}
 \|\nabla\eta^{0,\pm}(\cdot)\|_{L^2}^2=&\frac{b-a}{2}
 \sum\limits_{l=1}^{M-1}|\mu_l|^2|(\widehat{\eta^0})_l|^2\lesssim
  \tau^3\int_0^\tau\int_0^{s}\int_0^{s_1}\bigg(\|\nabla\p_{tt}\Tf(\psi(s_2))\|_{L^2}^2
\bigg)\,ds_2ds_1ds\lesssim \tau^6\vep^{2\alpha^\ast-4}.
 \end{split}
 \end{equation*}
 Thus, we have proved that $\|\eta^{0,\pm}(\cdot)\|_{L^2}+\|\nabla\eta^{0,\pm}(\cdot)\|_{L^2}\lesssim \tau^3\vep^{2-\alpha^\ast}$. Using the same approach (omitted here for brevity),
 we can deduce that
 \be
 \|\eta^{n,\pm}(\cdot)\|_{L^2}+\|\nabla\eta^{n,\pm}(\cdot)\|_{L^2}\lesssim \tau^3\vep^{\alpha^\ast-2},\qquad \text{for}\quad n\ge1.
 \ee
 Now, we are ready to prove the lemma. By combining
 all the results above and triangle inequality, we finally prove (\ref{xi:well}) and (\ref{xi:ill})
 if we separate the cases $\alpha\ge2$ and $\alpha\in[0,2)$.
\end{proof}

Now, let us look back at our main results.
Making use of explicit property of the proposed scheme,
we know that if the error estimates are correct, we can
replace the nonlinearity $\Tf(z)=f(|z|^2)z$ by a cut-off
nonlinearity $\Tf_{_B}(z)$ defined as
\begin{equation}\label{prop}
\Tf_{_B}(z)=\rho\left(\frac{|z|^2}{B}\right) f(|z|^2)z,\quad \text{where}\quad \rho(s)\in C^\infty_0(\Bbb R) \quad\text{satisfying}\quad\rho(s)=\begin{cases}1,&|s|\leq 1,\\
0,&|s|\ge2,\\
\in[0,1],&|s|\in[1,2].
\end{cases}
\end{equation}
where $B=(M_1+1)^2$ with $M_1$ given in (\ref{M1}).
It is true that this replacement doesn't affect NLSW
and NLS themselves. Thus, if the same error estimates
in Theorem \ref{thm:main} and \ref{thm:main2} hold for
NLSW with the truncated nonlinearity $\Tf_{_B}(z)$, it
is a direct consequence that the  error estimates
(Theorem \ref{thm:main} and \ref{thm:main2}) hold true
for NLSW with nonlinearity $\Tf(z)$ as  the corresponding
numerical schemes coincide in this case (see more discussions
in \cite{Bao0,Baoc}).

Based on the above observation, we only need to prove
Theorems \ref{thm:main} and \ref{thm:main2} for the
nonlinearity $\Tf(z)$ replaced by $\Tf_{_B}(z)$. From now
on, we will treat $\Tf(z)$ as $\Tf_{_B}(z)$ while not changing
the notation.

\bigskip

{\it Proof of  Theorem \ref{thm:main}.} Let us define the error
$e^n\in Y_M$ and $e^n(x)\in X_M$ ($n\ge0$) as
\begin{equation}\begin{split}
&e_j^n=(P_M\psi(t_n))(x_j)-\psi_j^n,\quad j\in\Tt_M^0,\qquad n\ge0\\
&e^n(x)=I_M(e^n)(x)=\sum\limits_{l=1}^{M-1}(\widetilde{e^n})_l\Phi_l(x),\quad n\ge0,\quad x\in \Omega.
\end{split}
\end{equation}
Then, by triangle inequality and Lemma \ref{lem:sine} as
well as assumptions (A) and (B), we obtain
\begin{equation}\label{eq:l2h1}\begin{split}
&\|\psi(\cdot,t_n)-\psi_I^n(\cdot)\|_{L^2}\leq\|\psi(\cdot,t_n)-P_M(\psi(t_n))(\cdot)\|_{L^2}+\|e^n(\cdot)\|_{L^2}
\lesssim h^{m_0}+ \|e^n(\cdot)\|_{L^2},\\
&\|\nabla[\psi(\cdot,t_n)-\psi_I^n(\cdot)]\|_{L^2}\lesssim h^{m_0-1}+\|\nabla e^n(\cdot)\|_{L^2},\quad n\ge0.
\end{split}
\end{equation}
Applying inverse inequality,  and Lemma \ref{lem:sine}
further, we have
\begin{equation}\label{eq:linfinity}
\begin{split}
\|\psi^n\|_{l^\infty}\leq &\sup\limits_{j\in\Tt_M^0}|\psi(x_j,t_n)-\psi_j^n| +\|\psi(t_n)\|_{L^\infty}
\leq \sup\limits_{j\in\Tt_M^0}|e^n_j+\psi(x_j,t_n)-P_M(\psi(t_n))(x_j)|+M_1\\
\leq &M_1+\|e^n\|_{l^\infty}+\sup\limits_{j\in\Tt_M^0}|I_M(\psi(t_n))(x_j)-P_M(\psi(t_n))(x_j)|
\\
\leq& M_1+\|e^n\|_{l^\infty}+C_1h^{-d/2}\|I_M(\psi(t_n))(\cdot)-P_M(\psi(t_n))(\cdot)\|_{L^2}\\
\leq& M_1+C_2h^{m_0-d/2}+\|e^n\|_{l^\infty}\leq M_1+\frac12+\|e^n\|_{l^\infty},\qquad n\ge0,
\end{split}
\end{equation}
where $C_1$ and $C_2$ depend on $\|\psi\|_{L^\infty([0,T];H_s^{m_0})}$
and $0<h<h_1$ for some $h_1>0$, $d$ is the dimension of the
spatial space, i.e., $d=1$ in the current case. However, we
put $d$ here to indicate how the proof  works  for  two and
three dimensions ($d=2,3$)  (see also Remark \ref{rmk:ext}).
As a consequence of the discrete Sobolev inequality in 1D,
\be
\|e^n\|_{l^\infty}^2\lesssim \|e^n\|_{l^2}\|\delta_x^+e^n\|_{l^2}\lesssim \|e^n(\cdot)\|_{L^2}\|\nabla e^n(\cdot)\|_{L^2},
\ee
we only need to estimate the $L^2$ and semi-$H^1$ norms of $e^n(x)$ ($n\ge0$).

{\it For $e^0$ and $e^1$.}
Considering $e^0(x)$, we note
\be
e^0(x)=P_M(\psi_0)(x)-I_M(\psi_0)(x),
\ee
in view of Lemma \ref{lem:sine} and (\ref{eq:interr}), we
get  $\|e^0(\cdot)\|_{L^2}\lesssim h^{m_0}\lesssim h^m$ and $\|\nabla e^0(\cdot)\|_{L^2}\lesssim h^{m_0-1}
\lesssim h^{m-1}$, where $m=\min\{m_0,k\}$.
For $e^1(x)$, we have $\psi_1^\vep\in H^m_{s}$ and
\be
 (\widetilde{e^1})_l=-(\widetilde{\xi^0})_l+c_l^0\left((\widehat{\psi_0})_l-(\widetilde{\psi_0})_l\right)
 +d_l^0\left((\widehat{\psi_1^\vep})_l-(\widetilde{\psi_1^\vep})_l\right).
\ee
Then Lemma \ref{lem:localerr} implies $\|e^1(\cdot)\|_{L^2}\lesssim h^{m}+\tau^3$
and $\|\nabla e^1(\cdot)\|_{L^2}\lesssim h^{m-1}+\tau^3$. It is
obvious $\|\psi^0\|_{l^\infty}\leq \|\psi_0\|_{L^\infty}\leq M_1+1$, and
by discrete Sobolev inequality, we obtain
\be
\|e^1\|_{l^\infty}^2\lesssim \|e^1(\cdot)\|_{L^2}\|\nabla e^1(\cdot)\|_{L^2}\lesssim h^{2m-2}+\tau^6,
\ee
which implies there exist $h_2,\tau_1>0$ such that
$\|e^1_j\|_{l^\infty}\leq \frac12$, and $\|\psi^1\|_{l^\infty}\leq M_1+1$ for $0<\tau\leq \tau_1$,
$0<h\leq h_2$ in view of (\ref{eq:linfinity}). This  proves
the conclusion for $n=0,1$ in Theorem \ref{thm:main} by noticing (\ref{eq:l2h1}).
Before going to the next step, we note the decomposition as
\be\label{eq:decom1}
(\widetilde{e^0})_l=(\widetilde{e^{0,+}})_l+(\widetilde{e^{0,-}})_l,\quad (\widetilde{e^1})_l=(\widetilde{e^{0,+}})_le^{i\tau\beta_l^+}+(\widetilde{e^{0,-}})_le^{i\tau\beta_l^-}-
(\widetilde{\xi^0})_l,
\quad l\in \Tt_M,
\ee
where $e^{0,\pm}(x)=\sum\limits_{l=1}^{M-1}(\widetilde{e^{0,\pm}})_l\Phi_l(x)\in X_M$
are given by
\begin{equation*}
(\widetilde{e^{0,+}})_l=-\frac{\beta_l^-}{\bd_l}[(\widehat{\psi_0})_l-(\widetilde{\psi_0})_l]-
\frac{i}{\bd_l}[(\widehat{\psi_1^\vep})_l
-(\widetilde{\psi_1^\vep})_l],\quad(\widetilde{e^{0,-}})_l
=\frac{\beta_l^+}{\bd_l}[(\widehat{\psi_0})_l-(\widetilde{\psi_0})_l]
+\frac{i}{\bd_l}[(\widehat{\psi_1^\vep})_l
-(\widetilde{\psi_1^\vep})_l].
\end{equation*}
We can easily derive that
\be
\|e^{0,\pm}(\cdot)\|_{L^2}\lesssim h^{m},\quad\text{and}\quad \|\nabla e^{0,\pm}(\cdot)\|_{L^2}\lesssim h^{m-1}.
\ee

\bigskip

{\it Error equation for $e^n(x)$ ($n\ge2$).}
For other time steps, subtracting (\ref{xidef})
from (\ref{scheme2:2}) and noticing
$e^{i\frac{\tau}{\vep^2}}=e^{i\tau\beta_l^+}\cdot e^{i\tau\beta_l^-}$ and $2e^{\frac{i\tau}{2\vep^2}}\cos(\frac{\tau\bd_l}{2})=e^{i\tau\beta_l^+}+e^{i\tau\beta_l^-}$,
we obtain
\begin{equation}\label{eq:errorl}\begin{split}
(\widetilde{e^{n+1}})_l=-e^{i\tau(\beta_l^++\beta_l^-)}(\widetilde{e^{n-1}})_l
+(e^{i\tau\beta_l^+}+e^{i\tau\beta_l^-})(\widetilde{e^{n}})_l-(\widetilde{\xi^{n}})_l+\widetilde{W}^n_l, \qquad n\ge1,\quad l\in\Tt_M,
\end{split}
\end{equation}
where $\widetilde{W}^n_l$ can be written as follows (similar as Lemma \ref{lem:localerr}),
\begin{equation}\label{eq:wnl}\begin{split}
\widetilde{W}^n_l=&e^{i\tau\beta_l^+}\widetilde{W}^{n,+}_l-e^{i\tau\beta_l^-}\widetilde{W}^{n,-}_l
-e^{\frac{i\tau}{\vep^2}}(\widetilde{W}^{n-1,+}_l-\widetilde{W}_l^{n-1,-}),\quad n\ge1,\quad\text{with}\quad\widetilde{W}_l^{0,\pm}=0,\\
\widetilde{W}^{n,\pm}_l=&p_l^{\pm}\left((\widetilde{\Tf(\psi(t_n))})_l-(\widetilde{\Tf(\psi^n)})_l\right)
+q_l^{\pm}\left(\delta_t^-(\widetilde{\Tf(\psi(t_n))})_l-\delta_t^-(\widetilde{\Tf(\psi^n)})_l\right),\quad n\ge1,\,l\in\Tt_M.
\end{split}
\end{equation}
{\it Property of $\widetilde{W}_l^{n,\pm}$.}
Next, we claim that if $\|\psi^n\|_{l^\infty}\leq M_1+1$, under assumptions (A) and (B), $f\in C^3([0,\infty))$ and $\Tf(z)$ enjoys the properties of $\Tf_{_B}(z)$ (\ref{prop}),  then we have (see detailed proof in Appendix)
\begin{equation}\label{nonlp}\begin{split}
&\frac{b-a}{2}\sum\limits_{l=1}^{M-1}|\widetilde{W}_l^{n,\pm}|^2\lesssim \tau^2\sum\limits_{k=n-1}^n\|e^k(\cdot)\|_{L^2}^2+\tau^2h^{2m_0},\\
&\frac{b-a}{2}\sum\limits_{l=1}^{M-1}\mu_l^2|\widetilde{W}^{n,\pm}|^2\lesssim \tau^2\sum\limits_{k=n-1}^n\left(\|e^k(\cdot)\|^2_{L^2}+\|\nabla[e^k(\cdot)]\|_{L^2}^2\right)+\tau^2h^{2m_0-2},\qquad n\ge1.
\end{split}
\end{equation}

{\it Proof by energy method.}  Using (\ref{eq:errorl}) iteratively, noticing Lemma \ref{lem:localerr} and above decomposition (\ref{eq:wnl}), we can find that for $n\ge1$,
\begin{align}\label{eq:induc1}
(\widetilde{e^{n+1}})_l=&(\widetilde{e^{0,+}})_le^{i(n+1)\tau\beta_l^+}+(\widetilde{e^{0,-}})_l
e^{i(n+1)\tau\beta_l^-}
-\sum\limits_{k=0}^{n}\left((\widetilde{\xi^{k,+}})_le^{i(n+1-k)\tau\beta_l^+}-
(\widetilde{\xi^{k,-}})_le^{i(n+1-k)\tau\beta_l^-}\right)\nn\\
&+\sum\limits_{k=1}^n\left(\widetilde{W}^{k,+}_le^{i(n+1-k)\tau\beta_l^+}-
\widetilde{W}^{k,-}_le^{i(n+1-k)\tau\beta_l^-}\right),\qquad l\in\Tt_M.
\end{align}
Using Cauchy inequality, we have
\begin{align}\label{eq:linf}
|(\widetilde{e^{n+1}})_l|^2
\leq &6 \biggl(|(\widetilde{e^{0,+}})_l|^2+|(\widetilde{e^{0,-}})_l|^2+
(n+1)\sum\limits_{k=0}^n|(\widetilde{\xi^{k,+}})_l|^2+
(n+1)\sum\limits_{k=0}^n|(\widetilde{\xi^{k,-}})_l|^2\nn\\
&+n\sum\limits_{k=1}^n|\widetilde{W}^{k,+}_l|^2+n\sum\limits_{k=1}^n|\widetilde{W}^{k,-}_l|^2
\biggl),\qquad l\in \Tt_M.
\end{align}
Then summing above inequalities together for $l\in\Tt_M$,  making use of Lemma \ref{lem:localerr} and Parseval identity, noticing  claim (\ref{nonlp}),  for all $1\leq n\leq \frac{T}{\tau}-1$, we have $\|e^{n+1}(\cdot)\|_{L^2}=\frac{b-a}{2}\sum\limits_{l=1}^{M-1}|(\widetilde{e^{n+1}})_l|^2$ and
\begin{align*}
\|e^{n+1}(\cdot)\|_{L^2}^2
\leq&\, 6\bigg(\|e^{0,+}(\cdot)\|_{L^2}^2+\|e^{0,-}(\cdot)\|_{L^2}^2
+(n+1)\sum\limits_{k=0}^n(\|\xi^{k,+}(\cdot)\|_{L^2}^2
+\|\xi^{k,-}(\cdot)\|_{L^2}^2)\\&+n C_0\tau^2\sum\limits_{k=0}^n(\|e^k(\cdot)\|_{L^2}^2+h^{2m_0})\bigg)\\
\leq& C_1h^{2m_0}+(n+1)^2\tau^2 C_2(\tau^4+h^{2m})
+C_0T\tau\sum\limits_{k=0}^n\|e^k(\cdot)\|_{L^2}^2\\
\leq& C_3(\tau^4+h^{2m})+C_0T\tau\sum\limits_{k=0}^n\|e^k(\cdot)\|_{L^2}^2,\quad 1\leq n\leq \frac{T}{\tau}-1,
\end{align*}
where $C_0$, $C_1$, $C_2$ and $C_3$ are constants depending on $T$, $f(\cdot)$, $M_1$ and
$\|\psi\|_{L^\infty([0,T];H_s^{m_0})}$.
Thus, the error bounds for $e^0(x)$ and $e^1(x)$ combined with discrete Gronwall inequality \cite{Glassey,Guo,Baoc} would imply that there exists a $\tau_2>0$ independent of $\vep$ such that when $0<\tau\leq\tau_2$, we can get
\be
\|e^{n+1}(\cdot)\|_{L^2}^2\lesssim \tau^4+h^{2m},\quad \text{and}
\quad \|e^{n+1}(\cdot)\|_{L^2}\lesssim \tau^2+h^m,\quad 1\leq n\leq \frac{T}{\tau}-1.
\ee

Then we estimate $\|\nabla e^{n+1}(\cdot)\|_{L^2}$. Similar as the above $L^2$ case, multiplying both sides of (\ref{eq:linf}) by $|\mu_l|^2$ and summing up for all $l\in\Tt_M$, noticing  claim (\ref{nonlp}), for all $1\leq n\leq  \frac{T}{\tau}-1$, using the derived $L^2$ estimates, we know
\begin{align*}
\|\nabla e^{n+1}(\cdot)\|_{L^2}^2
\lesssim& \bigg(\|\nabla e^{0,+}(\cdot)\|_{L^2}^2+\|\nabla e^{0,-}(\cdot)\|_{L^2}^2
+(n+1)\sum\limits_{k=0}^n(\|\nabla \xi^{k,+}(\cdot)\|_{L^2}^2
+\|\nabla\xi^{k,-}(\cdot)\|_{L^2}^2)\\&+n \tau^2\sum\limits_{k=0}^n(\|e^k(\cdot)\|_{L^2}^2+
\|\nabla e^k(\cdot)\|_{L^2}^2+h^{2m_0-2})\bigg)\\
\lesssim &(\tau^4+h^{2m-2})+\tau\sum\limits_{k=0}^n\|\nabla e^k(\cdot)\|_{L^2}^2,\quad 1\leq n\leq \frac{T}{\tau}-1.
\end{align*}
Again, the error bounds for $e^0(x)$ and $e^1(x)$ combined with the discrete Gronwall inequality \cite{Glassey,Guo,Baoc} would imply that there exists a $\tau_3>0$ independent of $\vep$ such that when $0<\tau\leq\tau_3$, we have
\be
\|\nabla e^{n+1}(\cdot)\|_{L^2}^2\lesssim \tau^4+h^{2m-2},\quad\text{and}
\quad \|\nabla e^{n+1}(\cdot)\|_{L^2}\lesssim \tau^2+h^{m-1},\quad 0\leq n\leq \frac{T}{\tau}-1.
\ee
It remains to show the $l^\infty$ bound of $\psi^{n+1}$.  The error bounds $\|e^{n+1}(\cdot)\|_{L^2}$ and $\|\nabla e^{n+1}(\cdot)\|_{L^2}$ and the discrete Sobolev inequality would imply that for some $h_3$, $\tau_4$ independent of $\vep$, if $0<h\leq h_3$, $0<\tau\leq\tau_4$
\be\label{eq:n+1}
\|e^{n+1}\|_{l^\infty}\lesssim \sqrt{\|e^{n+1}(\cdot)\|_{L^2}\|\nabla e^{n+1}(\cdot)\|_{L^2}}\lesssim h^{m-1}+\tau^2,
\quad\text{and}\quad \|\psi^{n+1}\|_{l^\infty}\leq M_1+1,
\ee
where $0\leq n\leq \frac{T}{\tau}-1$.
Now, in view of (\ref{eq:l2h1}) and (\ref{eq:n+1}), we see
\be
\|\psi(\cdot,t_{n})-\psi_I^{n}(\cdot)\|_{L^2}\lesssim h^m+\tau^2,\quad
\|\nabla[\psi(\cdot,t_{n})-\psi_I^{n}(\cdot)]\|_{L^2}\lesssim h^{m-1}+\tau^2,\quad 0\leq n\leq \frac{T}{\tau},
\ee
and this completes the proof, if we choose $0<\tau\leq\tau_0=\min\{\tau_1,\tau_2,\tau_3,\tau_4\}$
and $0<h\leq h_0=\min\{h_1,h_2,h_3\}$.
$\hfill\Box$

\section{Convergence in the ill-prepared initial data case}

Next, we start to prove Theorem \ref{thm:main2}. Again we will treat $\Tf(z)$ as $\Tf_{_B}(z)$ while not changing the notation.
The idea of the proof is shown in the diagram (\ref{chart}), where $\psi_{h,\tau}$ denotes the EWI-SP numerical solution,
$\psi$ is the exaction solution of the NLSW and $\psi^s$ is the exact solution of the corresponding NLS. In fact, the similar idea
has been employed in the study of asymptotic preserving (AP) schemes \cite{Deg,Jin}.

\be\label{chart}
\xymatrixcolsep{6pc}\xymatrix{
\psi_{h,\tau} \ar[r]^{O(h^{m}+\tau^2/\vep^{2-\alpha})} \ar[rd]^{\quad\quad L^2 \text{ error}}_{O(h^m+\tau^2+\vep^2)\quad}&
\psi\ar[d]^{O(\vep^2)} \\
&\psi^s}
\ee

{\it Proof of (\ref{ill-res1}) in Theorem \ref{thm:main2}.} First, let us consider  (\ref{ill-res1}), for which the proof is identical to the above proof for Theorem \ref{thm:main2}. The only thing needed to be modified is the local error bound, where $\|\xi^{n,\pm}(\cdot)\|_{L^2}\lesssim \tau(h^m+\vep^{\alpha-2}\tau^2)$ and  $\|\nabla\xi^{n,\pm}(\cdot) \|_{L^2}\lesssim \tau(h^{m-1}+\vep^{\alpha-2}\tau^2)$ (Lemma \ref{lem:localerr}) in this case. Following the analogous proof of Theorem \ref{thm:main2}, one can easily prove assertion (\ref{ill-res1}) and we omit the details here for brevity. $\hfill\Box$

Now, we come to prove (\ref{ill-res2}), using a similar idea in \cite{Baoc}. The proof is also similar to that of Theorem \ref{thm:main2}, and we just outline the main steps.  Let  $\psi^s:=\psi^s(x,t)$ be the solution of NLS (\ref{eq:nls}), and we write $\psi^s(t_n)$ for $\psi^s(x,t_n)$ in short. Denote the `error' $\Te^n\in Y_M$ and $\Te^n(x)\in X_M$ ($n\ge0$) as
\begin{equation}\begin{split}
&\Te_j^n=(P_M\psi^s)(x_j,t_n)-\psi_j^n,\quad j\in\Tt_M^0,\quad n\ge0,\\
&\Te^n(x)=I_M(\Te^n)(x)=\sum\limits_{l=1}^{M-1}(\widetilde{\Te^n})_l\Phi_l(x),\quad n\ge0,\quad x\in \Omega.
\end{split}
\end{equation}
Using triangle inequality, we have for $n\ge0$,
\begin{align}
\|\psi(\cdot,t_n)-\psi_I^n(\cdot)\|_{L^2}\leq &\,\|\psi(\cdot,t_n)-\psi^s(\cdot,t_n)\|_{L^2}
+\|\psi^s(\cdot,t_n)-P_M(\psi^s(t_n))(\cdot)\|_{L^2}+\|\Te^n(\cdot)\|_{L^2}\nn\\
\lesssim&\, \vep^2+h^{m_0}+\|\Te^n(\cdot)\|_{L^2},\label{eq:main2:1}\\
\|\nabla[\psi(\cdot,t_n)-\psi_I^n(\cdot)]\|_{L^2}\leq &\,\|\nabla[\psi(\cdot,t_n)-\psi^s(\cdot,t_n)]\|_{L^2}
+\|\nabla[\psi^s(\cdot,t_n)-P_M(\psi^s(t_n))(\cdot)]\|_{L^2}+\|\nabla\Te^n(\cdot)\|_{L^2}\nn\\
\lesssim&\, \vep^2+h^{m_0}+\|\nabla\Te^n(\cdot)\|_{L^2},\label{eq:main2:2}
\end{align}
which indicates that we only need to estimate $\|\Te^n(\cdot)\|_{L^2}$ and $\|\nabla\Te^n(\cdot)\|_{L^2}$.

In order to estimate the `error' $\Te^n$, we define the `local truncation error'  $\chi^n(x)=\sum\limits_{l=1}^{M-1}(\widetilde{\chi^n})_l\Phi_l(x)\in X_M$ for $\psi^s(x,t)$ as
\begin{align}\label{chidef}
(\widetilde{\chi^{0}})_l=&c_l^0(\widehat{\psi_0})_l+ d_l^0(\widehat{\psi_1})_l
+p_l^0(\widetilde{\Tf(\psi_0)})_l
+q_l^0(\widetilde{\Td(\psi^s(0))})
-(\widehat{\psi^s(\tau)})_l,\nn\\
(\widetilde{\chi^n})_l=&c_l(\widehat{\psi^s(t_{n-1})})_l+ d_l(\widehat{\psi^s(t_n)})_l
+p_l(\widetilde{\Tf(\psi^s(t_n))})_l
+q_l(\widetilde{\Td(\psi^s(t_n))})_l
-p_l^{\ast}(\widetilde{\Tf(\psi^s(t_{n-1}))})_l\\&
-q_l^{\ast}(\widetilde{\Td(\psi^s(t_{n-1}))})
-(\widehat{\psi^s(t_{n+1})})_l,\quad n\ge1.\nn
\end{align}
Rewriting the NLS (\ref{eq:nls}) as
\be
i\p_t\psi^s-\vep^2\p_{tt}\psi^s+\p_{xx}\psi^s+\left(f(|\psi^s|^2)\psi^s+\vep^2\p_{tt}\psi^s\right)=0,
\ee
and recalling $\psi^s(x,0)=\psi_0$ and $\p_t\psi^s(x,0)=\psi_1(x)$, we can prove analogous results in the same way as Lemma \ref{lem:localerr}. Here we present the results and  omit the details.
\begin{lemma}\label{lem:localerr2} Let $\chi^n(x)$ ($n\ge0$) be defined as (\ref{chidef}), $f(\cdot)\in C^k([0,\infty))$ ($k\ge3$), under assumptions {\rm (A)} and {\rm (B)}, we have the following decomposition of $\chi^n(x)$ ($n\ge0$) (\ref{eq:xideco}),
 \begin{equation}\label{deco2}\begin{split}
 &\chi^0(x)=e^{i\tau\beta_l^+}\chi^{0,+}(x)-e^{i\tau\beta_l^-}\chi^{0,-}(x),\qquad\\
 &\chi^n(x)=e^{i\tau\beta_l^+}\chi^{n,+}(x)-e^{i\tau\beta_l^-}\chi^{n,-}(x)
 -e^{\frac{i\tau}{\vep^2}}\left(\chi^{n-1,+}(x)-\chi^{n-1,-}(x)\right),\quad n\ge1,
 \end{split}
 \end{equation}
 where $\chi^{n,\pm}(x)=\sum\limits_{l=1}^{M-1}(\widetilde{\chi^{n,\pm}})_l\Phi_l(x)\in X_M$  and
 for $\alpha\in[0,2)$, i.e. the ill-prepared initial data case, we have
\begin{equation*}
\|\chi^{n,\pm}(\cdot)\|_{L^2}\lesssim \tau h^{m}+\tau^3+\tau\vep^2,\quad \|\nabla\chi^{n,\pm}(\cdot)\|_{L^2}\lesssim \tau h^{m-1}+\tau^3+\tau\vep^2,\quad m=\min\{m_0,k\},\quad n\ge0.
\end{equation*}
\end{lemma}

The decomposition in Lemma \ref{lem:localerr2} is analogous to that in the proof of Lemma \ref{lem:localerr} and the lemma also holds true for $\alpha\ge2$\,.

\bigskip

{\it Proof of (\ref{ill-res2}) and (\ref{ill-res3}) in Theorem \ref{thm:main2}.}  First, for $\Te^0(x)$ and $\Te^1(x)$ we note that we can write
\be
(\widetilde{\Te^0})_l=(\widetilde{\Te^{0,+}})_l+(\widetilde{\Te^{0,-}})_l,\quad (\widetilde{\Te^1})_l=(\widetilde{\Te^{0,+}})_le^{i\tau\beta_l^+}+(\widetilde{\Te^{0,-}})_le^{i\beta_l^-}-
(\widetilde{\chi^0})_l,
\quad l\in \Tt_M,
\ee
where $\Te^{0,\pm}(x)=\sum\limits_{l=1}^{M-1}(\widetilde{\Te^{0,\pm}})_l\Phi_l(x)\in X_M$ are given by
\begin{equation*}
(\widetilde{\Te^{0,+}})_l=-\frac{\beta_l^-}{\bd_l}[(\widehat{\psi_0})_l-(\widetilde{\psi_0})_l]-
\frac{i}{\bd_l}[(\widehat{\psi_1})_l
-(\widetilde{\psi_1^\vep})_l],\quad(\widetilde{\Te^{0,-}})_l
=\frac{\beta_l^+}{\bd_l}[(\widehat{\psi_0})_l-(\widetilde{\psi_0})_l]+\frac{i}{\bd_l}[(\widehat{\psi_1})_l
-(\widetilde{\psi_1^\vep})_l].
\end{equation*}
In view of the fact that $\frac{1}{\bd_l}\leq \vep^2$ and $\psi_1^\vep=\psi_1+\vep^\alpha\omega^\vep(x)$, it is easy to verify that
\be
\|\Te^{0,\pm}(\cdot)\|_{L^2}\lesssim h^{m}+\vep^{2+\alpha},\quad \|\nabla\Te^{0,\pm}(\cdot)\|_{L^2}\lesssim h^{m-1}+\vep^{2+\alpha},
\ee
which implies (\ref{ill-res2}) for $n=0,1$ in view of Lemma \ref{lem:localerr2}. For time steps $n\ge2$, following the same procedure of the proof for Theorem \ref{thm:main}, we can get conclusion  (\ref{ill-res2}) for all $0\leq n\leq\frac{T}{\tau}$ with the help of Lemma \ref{lem:localerr2}.

Finally, it remains to prove (\ref{ill-res3}) which enables us to simplify the nonlinearity $\Tf$ to $\Tf_{_B}$. Since assertions (\ref{ill-res1}) and (\ref{ill-res2}) have been proven and the constants in the estimates are independent of $\vep$, $\tau$ and $h$, we take the minimum of $\vep^2$ and $\frac{\tau^2}{\vep^{2-\alpha}}$ for $0<\vep\leq1$, and the following holds for $0\leq n\leq\frac{T}{\tau}$,
\begin{equation*}
\|\psi(\cdot,t_n)-\psi_I^n(\cdot)\|_{L^2}\lesssim h^{m}+\tau^{4/(2-\alpha)},\quad\|\nabla[\psi(\cdot,t_n)-\psi_I^n(\cdot)] \|_{L^2}\lesssim h^{m-1}+\tau^{4/(2-\alpha)}.
\end{equation*}
Hence  for $\alpha\in[0,2)$, Sobolev inequality   implies that
\be
\|\psi(\cdot,t_n)-\psi_I^n(\cdot)\|_{L^\infty}\leq \sqrt{\|\psi(\cdot,t_n)-\psi_I^n(\cdot)\|_{L^2}\|\nabla[\psi(\cdot,t_n)-\psi_I^n(\cdot)] \|_{L^2}}
\lesssim h^{m-1}+\tau,
\ee
which justifies $\max_{j\in{\cal T}_{M}}|\psi(x_j,t_n)-\psi_j^n|\leq 1$ for sufficient small $\tau$ and $h$. Hence (\ref{ill-res3}) is proven. The proof of Theorem \ref{thm:main2} is complete.
$\hfill\Box$

\begin{remark}\label{rmk:ext}Let us make a final remark that the above proof and the main results (Theorem \ref{thm:main} and \ref{thm:main2}) are valid for 2D and 3D cases. The key point for extension to 2D and 3D is the discrete Sobolev inequality in higher dimensions as \cite{Thomee}
 \be
 \|\psi_h\|_\infty\leq C|\ln h|\,\|\psi_h\|_{H^1},
 \qquad \|\phi_h\|_\infty\leq Ch^{-1/2}\|\phi_h\|_{H^1},
 \ee
 where $\psi_h$ and $\phi_h$  are 2D and 3D mesh functions with zero at the boundary,
  respectively,  and the discrete semi-$H^1$ norm $\|\cdot\|_{H^1}$ and $l^\infty$ norm
  $\|\cdot\|_\infty$ can be defined similarly  as the 1D version (\ref{norm}).  Thus, by requiring the time step $\tau$ satisfies the additional condition $\tau\lesssim h$,
   with the discrete Sobolev inequality and the uniform bounds for the semi-$H^1$ norm at $h^{m-1}+\tau$, using  $m\ge2$, we can control the $l^\infty$ bound of the numerical solution, which guarantees the correctness of the cutoff (\ref{prop}). The readers can refer to \cite{Baoc} for more discussion.
\end{remark}
\section{Numerical results}
In this section, we report the numerical results of our  scheme EWI-SP (\ref{scheme2:1})-(\ref{scheme2:2}) to confirm our theoretical analysis. The nonlinearity is taken as $f(|z|^2)=-|z|^2$.

\begin{table}[htb]
\begin{center}
\begin{tabular}{ccccc|cccc}\toprule
&\multicolumn{4}{|c|}{$\alpha=0$}&\multicolumn{4}{c}{$\alpha=2$}\\\cline{2-9}
&\multicolumn{1}{|c|}{$h=2$}&\multicolumn{1}{|c|}{$h=1$} & \multicolumn{1}{|c|}{$h=1/2$} & \multicolumn{1}{|c|}{$h=1/4$}
&\multicolumn{1}{|c|}{$h=2$}&\multicolumn{1}{|c|}{$h=1$} & \multicolumn{1}{|c|}{$h=1/2$} & $h=1/4$\\ \hline
$\vep=1/2$  &1.03E00&9.59E-2&6.76E-4&4.15E-8&9.45E-1&8.85E-2&6.39E-4&3.90E-8\\
$\vep=1/2^2$&8.85E-1&6.85E-2&2.36E-4&2.47E-9&8.55E-1&6.96E-2&2.45E-4&2.43E-9\\
$\vep=1/2^3$&8.62E-1&6.71E-2&1.57E-4&1.30E-10&8.55E-1&6.72E-2&1.57E-4&1.39E-10\\
$\vep=1/2^4$&8.57E-1&6.71E-2&1.40E-4&7.00E-11&8.60E-1&6.72E-2&1.40E-4&6.99E-11\\
$\vep=1/2^5$&8.61E-1&6.73E-2&1.37E-4&5.44E-11&8.62E-1&6.72E-2&1.37E-4&5.45E-11\\
$\vep=1/2^6$&8.62E-1&6.73E-2&1.36E-4&5.14E-11&8.62E-1&6.73E-2&1.36E-4&5.15E-11\\
$\vep=1/2^{10}$&8.62E-1&6.73E-2&1.36E-4&5.06E-11&8.62E-1&6.73E-2&1.36E-4&5.06E-11\\
$\vep=1/2^{20}$&8.62E-1&6.73E-2&1.36E-4&5.06E-11&8.62E-1&6.73E-2&1.36E-4&5.06E-11\\
\bottomrule
\end{tabular}
\end{center}
\caption{Spatial error analysis for  EWI-SP (\ref{scheme2:1})-(\ref{scheme2:2}), with different
$\vep$ for Case I ($\alpha=2$) and  Case II ($\alpha=0$),
for  $\|e^n(x)\|_{H^1}$.} \label{tab:spatial}
\end{table}

\begin{table}[htb]
\begin{center}
\begin{tabular}{ccccccccccc}\toprule
$\alpha=2$&$\tau=0.2$ & $\tau=\frac{0.2}{4}$ & $\tau=\frac{0.2}{4^2}$&$\tau=\frac{0.2}{4^3}$&$\tau=\frac{0.2}{4^4}$
&$\tau=\frac{0.2}{4^5}$&$\tau=\frac{0.2}{4^6}$&$\tau=\frac{0.2}{4^7}$\\ \hline
$\vep=1/2$&4.63E-2&2.97E-3&1.87E-4&1.17E-5&7.34E-7&4.59E-8&2.87E-9&1.87E-10\\
rate &--- &1.98&1.99&2.00&2.00&2.00&2.00&1.97\\ \hline
$\vep=1/2^2$&4.22E-2&4.52E-3&2.83E-4&1.77E-5&1.11E-6&6.91E-8&4.32E-9&2.77E-10\\
rate &--- &1.61&2.00&2.00&2.00&2.00&2.00&1.98\\ \hline
$\vep=1/2^3$&5.01E-2&3.99E-3&3.76E-4&2.37E-5&1.48E-6&9.27E-8&5.78E-9&3.53E-10\\
rate &--- &1.83&1.70&2.00&2.00&2.00&2.00&2.02\\ \hline
$\vep=1/2^4$&5.50E-2&3.73E-3&3.09E-4&2.45E-5&1.53E-6&9.61E-8&6.01E-9&3.74E-10\\
rate &--- &1.94&1.80&1.83&2.00&2.00&2.00&2.00\\ \hline
$\vep=1/2^5$&5.65E-2&3.85E-3&2.43E-4&1.95E-5&1.61E-6&1.02E-7&6.38E-9&4.04E-10\\
rate &--- &1.94&1.99&1.82&1.80&1.99&2.00&1.99\\ \hline
$\vep=1/2^6$&5.69E-2&3.88E-3&2.45E-4&1.54E-5&1.25E-6&1.00E-7&6.30E-9&3.82E-10\\
rate &--- &1.94&1.99&2.00&1.81&1.82&1.99&2.02\\ \hline
$\vep=1/2^{10}$&5.70E-2&3.89E-3&2.46E-4&1.54E-5&9.62E-7&6.01E-8&3.76E-9&2.48E-10\\
rate &--- &1.94&1.99&2.00&2.00&2.00&2.00&1.96\\ \hline
$\vep=1/2^{20}$&5.70E-2&3.89E-3&2.46E-4&1.54E-5&9.62E-7&6.01E-8&3.76E-9&2.46E-10\\
rate &--- &1.94&1.99&2.00&2.00&2.00&2.00&1.97\\
\bottomrule
\end{tabular}
\end{center}
\caption{Temporal error analysis for  EWI-SP, with different
$\vep$ for  Case I ($\alpha=2$),
with  $\|e^n(\cdot)\|_{H^1}$. The convergence rate
is calculated as $\log_2(\|e^n(\cdot,4\tau)\|_{H^1}/
\|e^n(\cdot,\tau)\|_{H^1})/2$.} \label{tab:temporal:2}
\end{table}

\begin{table}[htb]
\begin{center}
\begin{tabular}{ccccccccccc}\toprule
$\alpha=0$&$\tau=0.2$ & $\tau=\frac{0.2}{4}$ & $\tau=\frac{0.2}{4^2}$&$\tau=\frac{0.2}{4^3}$&$\tau=\frac{0.2}{4^4}$
&$\tau=\frac{0.2}{4^5}$&$\tau=\frac{0.2}{4^6}$&$\tau=\frac{0.2}{4^7}$\\ \hline
$\vep=1/2$&1.22E-1&7.58E-3&4.75E-4&2.97E-5&1.86E-6&1.16E-7&7.24E-9&4.42E-10\\
rate &--- &2.00&2.00&2.00&2.00&2.00&2.00&2.02\\ \hline
$\vep=1/2^2$&2.16E-1&2.61E-2&1.63E-3&1.02E-4&6.36E-6&3.98E-7&2.48E-8&1.55E-9\\
rate &--- &1.52&2.00&2.00&2.00&2.00&2.00&2.00\\ \hline
$\vep=1/2^3$&1.34E-1&6.07E-2&6.17E-3&3.89E-4&2.44E-5&1.52E-6&9.52E-8&5.91E-9\\
rate &--- &0.57&1.65&1.99&2.00&2.00&2.00&2.00\\ \hline
$\vep=1/2^4$&1.31E-1&1.12E-2&1.57E-2&1.54E-3&9.64E-5&6.04E-6&3.77E-7&2.34E-8\\
rate &--- &1.73&-0.20&1.67&2.00&2.00&2.00&2.00\\ \hline
$\vep=1/2^5$&1.32E-1&8.47E-3&1.66E-3&3.91E-3&3.86E-4&2.41E-5&1.51E-6&9.37E-8\\
rate &--- &1.98&1.18&-0.62&1.67&2.00&2.00&2.01\\ \hline
$\vep=1/2^6$&1.33E-1&8.09E-3&6.63E-4&3.61E-4&9.87E-4&9.66E-5&6.02E-6&3.75E-7\\
rate &--- &2.02&1.80&0.44&-0.73&1.68&2.00&2.00\\ \hline
$\vep=1/2^{7}$&1.34E-1&8.05E-3&5.01E-4&9.12E-5&8.78E-5&2.47E-4&2.42E-5&1.50E-6\\
rate &--- &2.03&2.00&1.23&0.03&-0.75&1.68&2.01\\ \hline
$\vep=1/2^{8}$&1.34E-1&8.05E-3&5.00E-4&3.64E-5&2.16E-5&2.20E-5&6.19E-5&6.01E-6\\
rate &--- &2.03&2.00&1.89&0.38&-0.01&-0.75&1.68\\ \hline
$\vep=1/2^{10}$&1.34E-1&8.06E-3&5.03E-4&3.17E-5&2.97E-6&2.02E-6&1.97E-6&1.99E-6\\
rate &--- &2.03&2.00&1.99&1.71&0.28&0.02&-0.01\\ \hline
$\vep=1/2^{20}$&1.34E-1&8.06E-3&5.02E-4&3.14E-5&1.96E-6&1.23E-7&7.66E-9&4.80E-10\\
rate &--- &2.03&2.00&2.00&2.00&2.00&2.00&2.00\\
\bottomrule
\end{tabular}
\end{center}
\caption{Temporal error analysis for  EWI-SP, with different
$\vep$ for Case II ($\alpha=0$),
with $\|e^n(\cdot)\|_{H^1}$.  The convergence rate
is calculated as $\log_2(\|e^n(\cdot,4\tau)\|_{H^1}/
\|e^n(\cdot,\tau)\|_{H^1})/2$.} \label{tab:temporal:0}
\end{table}

 \begin{table}[htb]
\begin{center}
\begin{tabular}{cccccccccc}\toprule
$\alpha=0$&$\ba{l}\vep=0.5\\ \tau=0.2\ea$&$\ba{l}\vep=0.5/2\\
\tau=0.2/4\ea$&$\ba{l}\vep=0.5/2^2\\ \tau=0.2/4^2\ea$ &
$\ba{l}\vep=0.5/2^3\\ \tau=0.2/4^3\ea$ & $\ba{l}\vep=0.5/2^4\\ \tau=0.2/4^4\ea$  & $\ba{l}\vep=0.5/2^5\\ \tau=0.2/4^5\ea$ \\ \hline
$\|e^n\|_{H^1}$&1.22E-1&2.61E-2&6.17E-3&1.54E-3&3.86E-4&9.66E-5\\
rate &--- &1.11&1.04&1.00&1.00&1.00\\\bottomrule
\end{tabular}
\end{center}
\caption{Degeneracy of convergence rate for EWI-SP,
{\it Case II} ($\alpha=0$). The convergence rate is calculated
as $\log_2(\|e^n(2^2\tau,2\vep)\|_{H^1}/\|e^n(\tau,\vep)\|_{H^1})/2$.} \label{tab3}
\end{table}

For the numerical experiments, the initial value is chosen as  $\psi_0(x)=\pi^{-1/4}e^{-x^2/2}$
 and $\omega^\vep(x)= e^{-x^2/2}$ in (\ref{eq:nls_wave}).
 The computational domain is chosen as $[a,b]=[-16,16]$. The `exact' solution is computed using the proposed scheme with  very fine mesh $h=1/128$ and
 time step $\tau=10^{-6}$.
  We study the following two cases of initial data:

{\it{Case I.}}  $\alpha=2$, i.e., the well-prepared initial data case.

{\it{Case II.}}  $\alpha=0$, i.e., the ill-prepared initial data case.

The errors are defined as $e^n\in Y_M$ and $e^n(x)\in X_M$ with $e_j^n=\psi(x_j,t_n)-\psi_j^n$  and $e^n(x)=\psi(t_n)-I_M(\psi^n)(x)$. We measure the $H^1$ norm of $e^n(x)$, i.e. $\|e^n(\cdot)\|_{H^1}=\|e^n(\cdot)\|_{L^2}+\|\nabla e^n(\cdot)\|_{L^2}$.

The errors are displayed at $t=1$. For spatial error analysis, we choose time step $\tau=10^{-6}$,  such that the error in time discretization can be neglected (cf. Tab. \ref{tab:spatial}). For temporal error analysis, we choose $h=1/32$ such that the spatial error can be ignored.

Tab. \ref{tab:spatial} depicts the spatial errors for both {\it{Case I}} and {\it{Case II}},
which clearly demonstrates that EWI-SP is uniformly  spectral accurate in $h$
for all $\vep\in(0,1]$. Tabs. \ref{tab:temporal:2} and \ref{tab:temporal:0}
present the temporal errors for {\it Case I} and {\it II}, respectively.
From Tab.  \ref{tab:temporal:2}, we can conclude that the temporal error
of EWI-SP is of second order uniformly in $\vep\in(0,1]$, for $\alpha=2$.
From Tab. \ref{tab:temporal:0} with $\alpha=0$, when time step $\tau$ is
small (upper triangle part of the table), second order convergence of
the temporal error is clear; when $\vep$ is small (lower triangle part of
the table), second order convergence of  the temporal error is also clear;
near the diagonal part where $\tau\sim \vep^2$,  degeneracy of
the convergence rate for the temporal error is observed. Tab \ref{tab3}
lists the degenerate convergence rate for  the temporal error with $\alpha=0$
at the parameter regime $\tau\sim \vep^2$, predicted by our error
estimates in Theorem \ref{thm:main2}. Numerical results clearly
confirms that  the temporal error of EWI-SP is of $O(\tau^2)$
and $O(\tau)$ uniformly in $\vep\in(0,1]$ for $\alpha=2$ and $\alpha=0$,
respectively, while  EWI-SP is uniformly spectral accurate in mesh
size $h$ for $\vep\in(0,1]$, for both well-prepared case $\alpha=2$ and
ill-prepared case $\alpha=0$.

\section{Conclusion}
We have proposed and analyzed an exponential wave integrator sine pseudospectral (EWI-SP) method   for the nonlinear Schr\"{o}dinger equation perturbed by the wave operator (NLSW) in one, two and three dimensions, where a  small dimensionless parameter $\vep\in(0,1]$ is used to describe the perturbation  strength. The  difficulty of the problem   is that the solution of NLSW oscillates in time at $O(\vep^2)$
wavelength with $O(\vep^4)$ and $O(\vep^2)$ amplitude  for well-prepared and ill-prepared
initial data, respectively, especially for $0<\vep\ll1$. We have proved the uniform spectral accuracy of EWI-SP in mesh size $h$, and uniform convergence rates of EWI-SP in time step $\tau$ at the order $O(\tau^2)$ and $O(\tau)$ for well-prepared and ill-prepared initial data, respectively, in $L^2$ norm and semi-$H^1$ norm. This improves the convergence results of finite difference methods for NLSW in \cite{Baoc}. Numerical results suggeste the error estimates are optimal.

\section*{Acknowledgments}
 Part of the work
 was  done while the authors were visiting the Institute for Mathematical
 Sciences, National University of Singapore,  in 2011.

\bigskip


{\centerline{\large\bf{Appendix: proof of the claim (\ref{nonlp})}}}
\renewcommand{\theequation}{{A}.\arabic{equation}}
\setcounter{equation}{0}

 We notice that $\Tf$ is assumed to enjoy the properties of $\Tf_B$ (\ref{prop}).
Employing Cauchy inequality and the fact about the coefficients $|p_l^\pm|\lesssim\tau$, $|q_l^\pm|\lesssim \tau^2$ (Lemma \ref{lem:localerr}), we can find
\begin{align*}
\sum\limits_{l=1}^{M-1}|\widetilde{W}^{n,\pm}_l|^2\lesssim&\,
\tau^2\left[
\sum\limits_{l=1}^{M-1}\left|(\widetilde{\Tf(\psi(t_n))})_l-(\widetilde{\Tf(\psi^n)})_l\right|^2
+\tau^2\sum\limits_{l=1}^{M-1}\left|\delta_t^-(\widetilde{\Tf(\psi(t_n))})_l
-\delta_t^-(\widetilde{\Tf(\psi^n)})_l\right|^2\right].
\end{align*}
Estimating each term on the RHS above, noticing the assumption which implies $\Tf(\cdot)$ is global Lipschitz, we get
\begin{align*}
&\frac{b-a}{2}\sum\limits_{l=1}^{M-1}\left|(\widetilde{\Tf(\psi(t_n))})_l-(\widetilde{\Tf(\psi^n)})_l\right|^2
=\|I_M(\Tf(\psi(t_n))(\cdot)-I_M(\Tf(\psi^n))(\cdot)\|_{L^2}^2\\&
=h\sum\limits_{j=1}^{M-1}|\Tf(\psi(x_j,t_n))-\Tf(\psi_j^n)|^2\leq C_{\Tf} h\sum\limits_{j=1}^{M-1}|\psi(x_j,t_n)-\psi_j^n|^2=C_{\Tf}\|I_M(\psi(t_n))(\cdot)-I_M(\psi^n)(\cdot)\|_{L^2}^2\\
&\leq 2C_{\Tf}\left(\|I_M(\psi(t_n))(\cdot)-P_M(\psi(t_n))(\cdot)\|_{L^2}^2+
\|P_M(\psi(t_n))(\cdot)-I_M(\psi^n)(\cdot)\|_{L^2}^2\right)\lesssim \|e^n(\cdot)\|_{L^2}^2+h^{2m_0},
\end{align*}
where $C_{\Tf}$ only depends on $f(\cdot)$ and $M_1$. Similarly, making use of the properties of $\Tf$, we have
\begin{align*}
&\frac{b-a}{2}\sum\limits_{l=1}^{M-1}\tau^2\left|\delta_t^-(\widetilde{\Tf(\psi(t_n))})_l
-\delta_t^-(\widetilde{\Tf(\psi^n)})_l\right|^2
=h\tau^2\sum\limits_{j=1}^{M-1}\left|\delta_t^-\Tf(\psi(x_j,t_n))-\delta_t^-\Tf(\psi_j^n)\right|^2\\&
\leq \widetilde{C}_{\Tf} h\left(\sum\limits_{j=1}^{M-1}|\psi(x_j,t_n)-\psi_j^n|^2
+\sum\limits_{j=1}^{M-1}|\psi(x_j,t_{n-1})-\psi_j^{n-1}|^2\right)
\leq  \widetilde{C}_{\Tf} h
\sum\limits_{k=n-1}^n\sum\limits_{j=1}^{M-1}|\psi(x_j,t_k)-\psi_j^k|^2\\
&= \widetilde{C}_{\Tf}\sum\limits_{k=n-1}^n\|I_M(\psi(t_k))(\cdot)-I_M(\psi^k)(\cdot)\|_{L^2}^2\lesssim \|e^n(\cdot)\|_{L^2}^2+\|
e^{n-1}(\cdot)\|_{L^2}^2+h^{2m_0},
\end{align*}
where $\widetilde{C}_{\Tf}$ only depends on $f(\cdot)$ and $M_1$.
Combining all the above results together, we get the first conclusion in the claim  (\ref{nonlp}).
Similarly, using Lemma \ref{lem:sine}, we can get
\begin{align*}
\sum\limits_{l=1}^{M-1}|\mu_l|^2|\widetilde{W}^{n,\pm}_l|^2
\lesssim&
\tau^2\left[
\|\nabla[I_M\left(\Tf(\psi(t_n))-\Tf(\psi^n)\right)]\,\|_{L^2}^2
+\tau^2\|\nabla[I_M\left(\delta_t^-\Tf(\psi(t_n))
-\delta_t^-\Tf(\psi^n)\right)]\,\|_{L^2}^2\right].
\end{align*}
In view of Lemma \ref{lem:sine}, the semi-$H^1$ norms on the RHS are equivalent to the discrete semi-$H^1$ norm of the corresponding grid functions, and we can estimate the RHS by estimating the  corresponding discrete semi-$H^1$ norms, which has been essentially done in our recent work \cite{Baoc}. For example, with the assumptions made in the claim (\ref{nonlp}),  we have
\begin{align*}
&\delta_x^+\left(\Tf(\psi(x_j,t_n))-\Tf(\psi_j^n)\right)\\&=
\int_0^1(G(\varsigma_\theta(t_n))\delta_x^+\psi(x_j,t_n)+H(\varsigma_\theta(t_n))
\overline{\delta_x^+\psi(x_j,t_n)})\,d\theta-\int_0^1
(G(\varsigma_\theta^n)\delta_x^+\psi_j^n+H(\varsigma_\theta^n)
\overline{\delta_x^+\psi_j^n})\,d\theta\\
&=\int_0^1\left[\left(G(\varsigma_\theta(t_n))-G(\varsigma_\theta^n)\right)\delta_x^+\psi(x_j,t_n)+
\left(H(\varsigma_\theta(t_n))-H(\varsigma_\theta^n)\right)
\overline{\delta_x^+\psi(x_j,t_n)}\right]\,d\theta\\
&\quad +\int_0^1\left[G(\varsigma_\theta^n)(\delta_x^+\psi(x_j,t_n)-\delta_x^+\psi_j^n)+
H(\varsigma_\theta^n)
(\overline{\delta_x^+\psi(x_j,t_n)}-\overline{\delta_x^+\psi_j^n})\right]\,d\theta,\quad j=0,1,\ldots,M-1,
\end{align*}
where $\varsigma_\theta(t_n)=\theta\psi(x_{j+1},t_n)+(1-\theta)\psi(x_j,t_n)$ and
$\varsigma_\theta^n=\theta\psi_{j+1}^n+(1-\theta)\psi_j^n$ for $\theta\in[0,1]$, $H(z)$ and $G(z)$ are defined in (\ref{scheme1:2}). Using the global Lipschitz properties of $G(\cdot)$ and $H(\cdot)$ in the current case, then it is a direct consequence that \cite{Baoc}
\be
|\delta_x^+\left(\Tf(\psi(x_j,t_n))-\Tf(\psi_j^n)\right)|\lesssim |\delta_x^+(\psi(x_j,t_n)-\psi_j^n)|
+\sum\limits_{k=j}^{j+1}|\psi(x_k,t_n)-\psi_k^n|,\quad 0\leq j\leq M-1,
\ee
which implies
\be\label{eq:h11}
\|\delta_x^+\left(\Tf(\psi(x_j,t_n))-\Tf(\psi_j^n)\right)\|_{l^2}^2\lesssim
\|\delta_x^+(\psi(x_j,t_n)-\psi_j^n)\|_{l^2}^2
+\|\psi(x_j,t_n)-\psi_j^n\|^2_{l^2}.
\ee
Combining (\ref{eq:h11}) together with Lemma \ref{lem:sine}, we have
\begin{align*}
\|\nabla[I_M\left(\Tf(\psi(t_n))-\Tf(\psi^n)\right)]\|_{L^2}^2\lesssim&
 \|\nabla[I_M(\psi(t_n)-\psi^n)(\cdot)]\|_{L^2}^2+\|I_M(\psi(t_n)-\psi^n)(\cdot)\|^2_{L^2}\\
\lesssim &\|\nabla e^n(\cdot)\|_{L^2}^2+\|e^n(\cdot)\|_{L^2}^2+h^{2m_0-2}.
\end{align*}
The remaining term can be estimated in the same way  and we omit the details here for brevity. Finally, we can obtain
\be
\frac{b-a}{2}\sum\limits_{l=1}^{M-1}|\mu_l|^2|\widetilde{W}^{n,\pm}_l|^2
\lesssim \tau^2\sum\limits_{k=n-1}^n\left(\|\nabla e^k(\cdot)\|_{L^2}^2+\|e^k(\cdot)\|_{L^2}^2\right)+\tau^2h^{2m_0-2}.
\ee
It is obvious  that the constant in the inequality only depends on $f(\cdot)$, $M_1$ and $\|\psi\|_{L^\infty([0,T];H_s^{m_0})}$.

\end{document}